\documentclass[11pt,reqno]{amsart}
\usepackage{graphicx}
\usepackage{amsmath,amsthm,amssymb,amsfonts,euscript,enumerate}
\usepackage{color,wrapfig}
\newtheorem{thm}{Theorem}[section]
\newtheorem{cor}[thm]{Corollary}
\newtheorem{lem}[thm]{Lemma}

\newcommand{\R}{{\mathbb{R}}}

\newcommand{\N}{{\mathbb{N}}}

\newcommand{\vp}{\varphi}

\newcommand{\La}{\triangle}
\newcommand{\la}{\triangle}
\newcommand{\bs}{\backslash}

\newcommand{\2}{\overline}
\newcommand{\3}{\varepsilon}
\newcommand{\4}{\widetilde}

\begin{document}

\title{Extinction profile of the logarithmic diffusion equation}

\author[Kin Ming Hui]{Kin Ming Hui}
\address{Kin Ming Hui: 
Institute of Mathematics, Academia Sinica,
Taipei, 10617, Taiwan, R.O.C.}
\email{kmhui@gate.sinica.edu.tw}

\author[Sunghoon Kim]{Sunghoon Kim}

\address{Sunghoon Kim:
Department of Mathematics and PMI (Pohang Mathematics Institute), Pohang University of Science and Technology (POSTECH), Hyoja-Dong San 31, Nam-gu, Pohang 790-784, South Korea} 
\email{ math.s.kim@postech.ac.kr}


\maketitle

%\author{Kin Ming Hui and Sunghoon Kim\\
%Institute of Mathematics, Academia Sinica\\
%Taipei, Taiwan, R.O.C.\\
%kmhui@gate.sinica.edu.tw, gauss79@math.sinica.edu.tw}
%\date{Nov 10, 2011}
%\smallbreak \maketitle
%\vspace{-0.3in}

\begin{abstract}
Let $u$ be the solution of $u_t=\Delta\log u$ in $\R^N\times (0,T)$, 
$N=3$ or $N\ge 5$, with initial value $u_0$ satisfying 
$B_{k_1}(x,0)\le u_0\le B_{k_2}(x,0)$ for some constants $k_1>k_2>0$ where 
$B_k(x,t)=2(N-2)(T-t)_+^{N/(N-2)}/(k+(T-t)_+^{2/(N-2)}|x|^2)$ is the 
Barenblatt solution for the equation. We give a new different proof on 
the uniform convergence of the rescaled function 
$\4{u}(x,s)=(T-t)^{-N/(N-2)}u(x/(T-t)^{-1/(N-2)},t)$, $s=-\log (T-t)$, 
on $\R^N$ to the rescaled Barenblatt solution
$\4{B}_{k_0}(x)=2(N-2)/(k_0+|x|^2)$ for some $k_0>0$ as $s\to\infty$.
We also obtain convergence of the rescaled solution $\4{u}(x,s)$ as $s\to\infty$
when the initial data satisfies $0\le u_0(x)\le B_{k_0}(x,0)$ in $\R^N$ 
and $|u_0(x)-B_{k_0}(x,0)|\le f(|x|)\in L^1(\R^N)$ for some constant
$k_0>0$ and some radially symmetric function $f$.

\end{abstract}

%\author[Kin Ming Hui]{Kin Ming Hui}
%\address{Kin Ming Hui:
%Institute of Mathematics, Academia Sinica,\\
%Taipei, Taiwan, R.O.C.}
%\email{kmhui@gate.sinica.edu.tw}
%\author[Sunghoon Kim]{Sunghoon Kim}
%\address{Sunghoon Kim:
%Institute of Mathematics, Academia Sinica,\\
%Taipei, Taiwan, R.O.C.}
%\email{gauss79@math.sinica.edu.tw}
%\keywords{}
%\subjclass{Primary 35B40 Secondary 35K15, 35K65}

\setcounter{equation}{0}
\setcounter{section}{0}

\section{Introduction}
\setcounter{equation}{0}
\setcounter{thm}{0}

The equation
\begin{equation}\label{u^m-cauchy-problem}
u_t=\la\phi_m(u)\quad\mbox{ in }\R^N\times (0,T)
\end{equation}
where $\phi_m(u)=u^m/m$ for $m\ne 0$ and $\phi_m(u)=\la\log u$
for $m=0$ arises in many physical models such as the flow of
gases through porous media \cite{A}, \cite{P}. When $m=1$, 
\eqref{u^m-cauchy-problem} is the heat equation. When $m=0$ and $N=1$,
the equation \eqref{u^m-cauchy-problem} arises as the limiting 
density distribution of two 
gases moving against each other and obeying the Boltzmann equation 
\cite{K}, as the diffusive limit for finite velocity Boltzmann 
kinetic models \cite{LT}, and in the model of viscous liquid film 
lying on a solid surface and subjecting to long range Van der Waals 
interactions with the fourth order term being neglected \cite{G}, \cite{WD}. 
When $m=0$ and $N=2$, \eqref{u^m-cauchy-problem} arises as the Ricci 
flow on the complete surface $\R^2$ \cite{W1}, \cite{W2}. We refer 
the reader to the book \cite{V3} by J.L.~Vazquez for the basics of 
the above equation and the books \cite{DK}, \cite{V2}, 
by P.~Daskalopoulos, C.E.~Kenig, and J.L.~Vazquez for the recent 
research results on \eqref{u^m-cauchy-problem}.

As observed by J.L.~Vazquez \cite{V1} there is a great difference in the 
behaviour of the solutions of \eqref{u^m-cauchy-problem} for 
$m>(N-2)_+/N$ and for $m\le (N-2)_+/N$. For example for $m>(N-2)_+/N$ 
there exists global $L^1(\R^N)$ solution of \eqref{u^m-cauchy-problem} while 
for $0<m\le (N-2)_+/N$ and $N\ge 3$ the $L^1(\R^N)$ solutions of 
\eqref{u^m-cauchy-problem} vanish in a finite time. For $m\le -1$ 
and $N=1$ there exists no finite mass solution of 
\eqref{u^m-cauchy-problem}. 

In \cite{DS1} P.~Daskalopoulos and N.~Sesum proved the convergence 
of the rescaled solution of \eqref{u^m-cauchy-problem} to the 
rescaled Barenblatt solution of \eqref{u^m-cauchy-problem} near 
the extinction time for the case $0<m\le (N-2)_+/N$, $N>2$, with 
initial data that behaves like $O(|x|^{-2/(1-m)})$ as $|x|\to\infty$. 
Extinction behaviour of the solution of 
\begin{equation}\label{eq-cases-main-cauchy-problem}
\begin{cases}
u_t=\La\log u\qquad\mbox{ in }\R^N\times(0,T),\\
u(x,0)=u_0(x)\quad\mbox{ in }\R^N
\end{cases}
\end{equation}
for the case $N=2$ was studied by S.Y.~Hsu \cite{Hs2}, \cite{Hs3},
P.~Daskalopoulos, M.A.~del Pino and N.~Sesum \cite{DP2}, \cite{DS2} 
and K.M.~Hui \cite{Hu3}.

In \cite{Hu2} K.M.~Hui proved that any solution of \eqref{eq-cases-main-cauchy-problem} 
with $N\geq 3$ and initial value satisfying the condition $0\le u_0(x)
\le C/|x|^2$ for all $|x|\ge R_0$ and some constants $R_0>0$, $C>0$,
will vanish in a finite time. It would be interesting to find the
extinction behaviour of the solution of 
\eqref{eq-cases-main-cauchy-problem} for the case $N\ge 3$.
In this paper we will study the asymptotic behaviour of solutions 
of \eqref{eq-cases-main-cauchy-problem} for $N=3$ and $N\ge 5$ near its 
extinction time under the assumption that the initial value $u_0$ 
is non-negative, locally integrable, and 
\begin{equation}\label{eq-growth-condition-of-initial-value}
u_0(x)\approx\frac{C}{|x|^2}\quad\mbox{ as }|x|\to\infty.
\end{equation}
Note that the self-similar Barenblatt solutions of 
\eqref{eq-cases-main-cauchy-problem} for $N\ge 3$ are given explicitly by
\begin{equation}\label{eq-self-similar-barenblett-solution}
B_{k}(x,t)=\frac{2(N-2)(T-t)_+^{\frac{N}{N-2}}}{k+(T-t)_+^{\frac{2}{N-2}}|x|^2},
\qquad k>0,
\end{equation}
which satisfy the growth condition 
\eqref{eq-growth-condition-of-initial-value}.

Note that the asymptotics of the solutions of the fast diffusion equation 
\eqref{u^m-cauchy-problem} for the case $0<m<1$ and the case $m=(N-4)/(N-2)$ 
is studied by A.~Blanchet, M.~Bonfort, J.~Dolbeault, G.~Grillo and 
J.L.~Vaquez in \cite{BBDGV} and \cite{BGV}. Sharp decay rate of the 
solutions of \eqref{u^m-cauchy-problem} for the case $0<m<1$ and 
$m\ne (N-4)/(N-2)$ is proved in \cite{BDGV}. A sketch that their proofs 
extended to the case $m=0$ is also given in appendix B of \cite{BBDGV}.
The proof in \cite{BBDGV}, \cite{BGV} and \cite{BDGV} used
Lypunov functional technique. 
On the other hand in this paper we will give a totally different proof of 
the asymptotics of the solutions of the fast diffusion equation 
\eqref{u^m-cauchy-problem} for the case $m=0$ and $N=3$ or $N>5$
near the extinction time using a modification of the potential technique 
of P.~Daskalopoulos and N.~Sesum \cite{DS1}. 

We will assume $N\ge 3$ for the rest of the paper.
We will also assume in the first part of this paper that the initial 
condition $u_0$ is trapped in between two Barenblatt solutions, i.e.,
\begin{equation}\label{eq-initial-u-0-trapped-by-barenblat}
B_{k_1}(x,0)\leq u_0(x)\leq B_{k_2}(x,0)
\end{equation}
for some constants $k_1>k_2>0$. We will consider first solutions of 
\eqref{eq-cases-main-cauchy-problem} which satisfy the condition
\begin{equation}\label{eq-initial-u-0-trapped-by-B}
B_{k_1}(x,t)\leq u(x,t) \leq B_{k_2}(x,t) \qquad \mbox{ in }\R^N\times (0,T). 
\end{equation}
Note that if $u$ is the maximal solution of 
\eqref{eq-cases-main-cauchy-problem} for $N\ge 3$ with initial value 
satisfying \eqref{eq-initial-u-0-trapped-by-barenblat}, then by the 
result of \cite{Hu2} $u$ satisfies \eqref{eq-initial-u-0-trapped-by-B}.

Consider the rescaled function
\begin{equation}\label{eq-rescaled-function}
\4{u}(x,s)=\frac{1}{(T-t)^{\frac{N}{N-2}}}u
\left(\frac{x}{(T-t)^{\frac{1}{N-2}}},t\right), \qquad s=-\log (T-t).
\end{equation}
By direct computation $\4{u}$ satisfies
\begin{equation}\label{eq-rescaled-main-equation-of-v}
\4{u}_s=\La\log \4{u}+\frac{1}{N-2}\mbox{div}
(x\cdot \4{u})\qquad \mbox{ in }\R^N\times (-\log T,\infty).
\end{equation}
By \eqref{eq-initial-u-0-trapped-by-B} and \eqref{eq-rescaled-function},
\begin{equation}\label{eq-initial-tilde-u-trapped-by-barenblat}
\4{B}_{k_1}(x)\leq \4{u}(x,s)\leq\4{B}_{k_2}(x)
\end{equation}
holds in $\R^N\times (-\log T,\infty)$ where
\begin{equation}\label{eq-rescaled-Barenblatt-23}
\4{B}_k(x)=\frac{2(N-2)}{k+|x|^2}.
\end{equation}
The main convergence results that we will prove in this paper are
the following.

\begin{thm}\label{lem-main-theorem-with-finite-integral-MS}
Let $N=3$ and let $u_0$ satisfy \eqref{eq-initial-u-0-trapped-by-barenblat} 
for some constant $k_1>k_2>0$. Suppose $u$ is a solution of 
\eqref{eq-cases-main-cauchy-problem} with initial value $u_0$ which 
satisfies \eqref{eq-initial-u-0-trapped-by-B}.
Then the rescaled function $\4{u}$ given by 
\eqref{eq-rescaled-function} converges uniformly on $\R^3$ and also 
in $L^1(\R^3)$ as $s\to\infty$ to the rescaled Barenblatt solution 
$\4{B}_{k_0}$ for some constant $k_0>0$ uniquely determined by 
\begin{equation}\label{initial-mean=0}
\int_{\R^N}(u_0(x)-B_{k_0}(x,0))\,dx=0.
\end{equation}
\end{thm}

\begin{thm}\label{lem-main-theorem-MS}
Let $N\ge 5$ and let $u$ be a solution of \eqref{eq-cases-main-cauchy-problem} 
with initial value $u_0$ satisfying 
\eqref{eq-initial-u-0-trapped-by-barenblat} 
and 
\begin{equation}\label{u0-B-in-l1} 
u_0=B_{k_0}+f
\end{equation}
for some constants $k_1>k_2>0$, $k_0>0$ and $f\in L^1(\R^N)$ where $B_{k_0}$ 
is the Barenblatt solution. Suppose $u$ satisfies 
\eqref{eq-initial-u-0-trapped-by-B}. Let $\4{u}$ be the rescaled 
function given by 
\eqref{eq-rescaled-function}. Then $\4{u}$ converges uniformly on $\R^N$ 
and in the weighted space $L^1(\4{B}^{\frac{N-4}{2}},\R^N)$ as $s\to\infty$ 
to the rescaled Barenblatt solution $\4{B}_{k_0}$.
\end{thm}

The plan of the paper is as follows. In section 2 we will establish 
some a priori estimates for the solutions of 
\eqref{eq-cases-main-cauchy-problem}. We will prove 
Theorem~\ref{lem-main-theorem-with-finite-integral-MS} 
and Theorem~\ref{lem-main-theorem-MS} in sections three and four
respectively. In section five we will improve Theorem~\ref{lem-main-theorem-MS}  
by removing the condition \eqref{eq-initial-u-0-trapped-by-barenblat} on the 
initial data.

We start with some definitions. We say that $u$ is a solution of 
\eqref{eq-cases-main-cauchy-problem} in $\R^N\times (0,T)$ if
$u>0$ in $\R^N\times (0,T)$ and $u$ satisfies 
\eqref{eq-cases-main-cauchy-problem} in the classical sense in 
$\R^N\times (0,T)$ with
$$
u(\cdot,t)\to u_0\quad\mbox{ in }L_{loc}^1(\R^N) \quad\mbox{ as }t\to 0.
$$
We say that $u$ is a maximal solution of 
\eqref{eq-cases-main-cauchy-problem} in $\R^N\times (0,T)$ if 
$u$ is a solution of 
\eqref{eq-cases-main-cauchy-problem} in $\R^N\times (0,T)$ and
$u\ge v$ for any solution $v$ of 
\eqref{eq-cases-main-cauchy-problem} in $\R^N\times (0,T)$. 
For any $R>0$ and $x_0\in\R^N$, let $B_R(x_0)=\{x\in\R^N:|x-x_0|<R\}$.
Let $\omega_N$ be the surface area of the unit sphere $S^{N-1}$ in $\R^N$.
For any $a\in\R$, let $a_{\pm}=\max (\pm a,0)$. We will assume $N\ge 3$ 
for the rest of the paper.

For any $\alpha>0$, we define the weighted $L^1$-space with weight 
$\4{B}^{\alpha}(x):=\left(\frac{2(N-2)}{k_2+|x|^2}\right)^{\alpha}$ as
\begin{equation*}
L^1(\4{B}^{\alpha},\R^N):=\left\{f\Big|\int_{\R^N}f(x)\4{B}^{\alpha}(x)
\,dx<\infty\right\}.
\end{equation*}

\section{Preliminary Estimates}
\setcounter{equation}{0}
\setcounter{thm}{0}

In this section we will establish some a priori estimates for the 
solutions of \eqref{eq-cases-main-cauchy-problem}.

\begin{lem}\label{lem-preliminary-esti-first-1}
Let $u$, $v$ be two solutions of \eqref{eq-cases-main-cauchy-problem} 
with initial values $u_0$, $v_0$ respectively. Assume in addition that 
$u$, $v\geq B$, for some Barenblatt solution $B=B_k$ given by 
\eqref{eq-self-similar-barenblett-solution}. Then there exists a 
constant $C>0$ such that 
\begin{equation*}
(i) \quad \left(\int_{B_{R}(x)}(u-v)_+(y,t)\,dy\right)^{\frac{1}{2}} \leq
 \left(\int_{B_{2R}(x)}(u_0-v_0)_+(y)\,dy\right)^{\frac{1}{2}}
+CR^{\frac{N-2}{2}}\sqrt{T}
\end{equation*} 
and 
\begin{equation*}
(ii) \quad \left(\int_{B_{R}(x)}|u-v|(y,t)\,dy\right)^{\frac{1}{2}} \leq 
\left(\int_{B_{2R}(x)}|u_0-v_0|(y)\,dy\right)^{\frac{1}{2}}+CR^{\frac{N-2}{2}}
\sqrt{T}
\end{equation*}
holds for any $R\ge |x|+\sqrt{k}\delta^{-\frac{1}{N-2}}$, $x\in\R^N$, 
$0<t\le T-\delta$, and $0<\delta<T$.
\end{lem}
\begin{proof}
We will use a modification of the argument of \cite{Hu2} to prove 
the lemma. Without loss of generality we may assume that $x=0$. Let 
$\eta\in C^{\infty}_0(\R^N)$, $0\leq\eta\leq 1$, be such that 
$\eta(x)=1$ for $|x|\leq 1$, $\eta=0$ for $|x|\geq 2$ and 
$\eta_R(x)=\eta(x/R)$ for any $R>0$. Then 
$|\La \eta_R|\leq \frac{C_1}{R^2}$ and $|\nabla \eta_R|\leq 
\frac{C_1}{R}$ for some constant $C_1>0$. By the Kato inequality 
\cite{K},
\begin{equation}\label{eq-Kato-inequality-1}
\frac{\partial}{\partial t}\int_{\R^N}(u-v)_+(x,t)\eta^4_R(x)\,dx 
\leq \int_{\R^N}(\log u-\log v)_+(x,t)\La\eta^4_R(x)\,dx\quad\forall 
0<t<T. 
\end{equation} 
Since $v\geq B_k$ for some Barenblatt solution $B_k$, 
\begin{equation}\label{eq-aligned-some-inequality-of-log-1}
(\log u-\log v)_+=\left(\log\left(\frac{u}{v}\right)\right)_+
\leq C\left(\frac{u}{v}-1\right)_+^{\frac{1}{2}}
\leq C\frac{(u-v)_+^{\frac{1}{2}}}{v_+^{\frac{1}{2}}}
\leq CB_k^{-\frac{1}{2}}(u-v)_+^{\frac{1}{2}}
\end{equation}
for some generic constant $C>0$. By \eqref{eq-Kato-inequality-1},
\eqref{eq-aligned-some-inequality-of-log-1}, and the H\"older 
inequality,
\begin{align}\label{eq-aligned-some-inequality-of-log-1245}
&\frac{\partial}{\partial t}\int_{\R^N}(u-v)_+(x,t)\eta^4_R(x)\,dx
\nonumber\\ 
\leq&C\int_{\R^N}(u-v)_+^{\frac{1}{2}}(x,t)B_k^{-\frac{1}{2}}(x)|
\La \eta^4_R|(x)\,dx\nonumber\\
\leq&C\left(\int_{\R^N}(u-v)_+(x,t)\eta^4_R(x)\,dx
\right)^{\frac{1}{2}}\left(\int_{R\leq |x|\leq 2R}\eta_R^{-4}B_k^{-1}|\La 
\eta^4_R|^2\,dx\right)^{\frac{1}{2}}\nonumber\\
\leq&C\left(\int_{\R^N}(u-v)_+(x,t)\eta^4_R(x)\,dx
\right)^{\frac{1}{2}}\left(\int_{R\leq |x|\leq 2R}B_k^{-1}
\left(32\eta_{R}^{2}|\La \eta_R|^2+288|\nabla \eta_R|^4\right)\,dx
\right)^{\frac{1}{2}}.
\end{align}
Since 
\begin{equation}\label{eq-condition-result-of-B--1}
(B_k(x,t))^{-1}\leq \frac{C|x|^{2}}{T-t}\quad\forall 
|x|\ge\sqrt{k}\delta^{-\frac{1}{N-2}}, 0\le t\le T-\delta, 0<\delta<T, 
\end{equation}
by \eqref{eq-aligned-some-inequality-of-log-1245} we have 
\begin{equation*}
\frac{\partial}{\partial t}\int_{\R^N}(u-v)_+(x,t)\eta^4_R(x)\,dx 
\leq C\frac{R^{\frac{N-2}{2}}}{(T-t)^{\frac{1}{2}}}
\left(\int_{\R^N}(u-v)_+(x,t)\eta^4_R(x)\,dx\right)^{\frac{1}{2}}
\end{equation*} 
for any $R^2\ge k\delta^{-\frac{2}{N-2}}$, $0<t\le T-\delta$, and
$0<\delta<T$. By integrating the above differential 
inequality with respect to $t$, we get $(i)$. Similarly, 
\begin{equation*}
\left(\int_{B_{R}(x)}(u-v)_-(x,t)\,dx\right)^{\frac{1}{2}} \leq 
\left(\int_{B_{2R}(x)}(u_0-v_0)_-(x,t)\,dx\right)^{\frac{1}{2}}
+CR^{\frac{N-2}{2}}\sqrt{T}.
\end{equation*} 
holds for any $R^2\ge k\delta^{-\frac{2}{N-2}}$, $0<t\leq T-\delta$, 
and $0<\delta<T$. $(ii)$ then follows by adding the above inequality 
with $(i)$.
\end{proof}

\begin{lem}\label{lem-before-L-1-contraction-46}
Let $u$, $v$ be two solutions of 
\eqref{eq-cases-main-cauchy-problem} with initial values 
$u_0$, $v_0$, respectively. Assume in addition that $u$, 
$v\ge B$, for some Barenblatt solution $B=B_k$ given by 
\eqref{eq-self-similar-barenblett-solution}. If 
$f=u_0-v_0\in L^1(\R^N)$, then $u(\cdot,t)-v(\cdot,t)\in L^1(\R^N)$ 
for all $t\in[0,T)$.
\end{lem}
\begin{proof}
We will use a modification of the proof of Lemma 2.1 of \cite{DS1} to prove 
the lemma. We introduce the potential function
\begin{equation*}
w(x,t)=\int_0^t|(\log u-\log v)(x,s)|\,ds\quad\forall 0<t\leq T-\delta.
\end{equation*}
By the Kato inequality \cite{K}, 
\begin{equation*}
\La|\log u-\log v|\geq \mbox{sign}(u-v)\La(\log u-\log v),
\end{equation*}
and so from equation \eqref{eq-cases-main-cauchy-problem}, we obtain
\begin{equation}\label{eq-E-by-H-1924}
\frac{\partial}{\partial t}|u-v|\leq \La|\log u-\log v|.
\end{equation}
Integrating the above inequality in time, and using that 
$|f|=|u_0-v_0|$, we obtain
\begin{equation}\label{eq-lowerbound-w=logu-logv-1}
\La w\geq -|f| \qquad \mbox{ in }\R^N\qquad\qquad\qquad\qquad\forall 0<t<T.
\end{equation}
Let
\begin{equation*}
Z(x)=\frac{1}{(N-2)\omega_N}\int_{\R^N}\frac{|f(y)|}{|x-y|^{N-2}}\,dy
\end{equation*}
denote the Newtonian potential of $|f|$ where $\omega_N$ is the surface 
area of the unit sphere $S^{N-1}$ in $\R^N$. Then by 
\eqref{eq-lowerbound-w=logu-logv-1},
\begin{equation}\label{eq-in-the-sense-of-distributions}
\La(w(\cdot,t)-Z)\geq 0 
\end{equation}
in the sense of distributions in $\R^N$ for any $0<t<T$. Next we 
would like to show that
\begin{equation}\label{eq-upperbound-newto-Z-A}
\int_{R\leq|x|\leq 2R}w(x,t)\,dx\leq \int_{R\leq|x|\leq 2R}Z(x)\,dx 
\qquad \forall R>0.
\end{equation}
In order to prove this estimate we first suppose that 
$f\in L^1(\R^N)\cap L^{\infty}(\R^N)$. By 
\eqref{eq-in-the-sense-of-distributions} and the mean value 
property for subharmonic functions,
\begin{equation}\label{eq-upperbound-newto-Z-1}
\begin{aligned}
w(x,t)&\leq Z(x)+\frac{N}{\omega_N\rho^N}
\int_{B_{\rho}(x)}\left(w(y,t)-Z(y)\right)\,dy\\
&\leq Z(x)+\frac{N}{\omega_N\rho^N}\int_{B_{\rho}(x)}w(y,t)\,dy
\end{aligned}
\end{equation}
holds for any $x\in\R^N$, $0<t<T$, and $\rho>0$. We claim that 
\begin{equation}\label{eq-lim-of-w-as-rho-to-infty}
\lim_{\rho\to\infty}\frac{1}{\rho^N}\int_{B_{\rho}(x)}w(y,t)\,dy=0 
\qquad \forall x\in\R^N, 0<t<T.
\end{equation}
In order to prove 
\eqref{eq-lim-of-w-as-rho-to-infty} it suffices to prove that
\begin{equation*}\label{eq-lim-of-w-1-w-2-as-rho-to-infty-result-1}
\lim_{\rho\to\infty}\frac{1}{\rho^N}I_1(\rho,t)=0 \quad \mbox{and} 
\quad \lim_{\rho\to\infty}\frac{1}{\rho^N}I_2(\rho,t)=0\qquad 
\forall 0<t<T
\end{equation*}
where
\begin{equation*}\left\{\begin{aligned}
&I_1(\rho,t)=\int_{0}^{t}\int_{B_{\rho}(x)}\left(\log u-\log v\right)_+
(y,s)\,dy\,ds\\
&I_2(\rho,t)=\int_{0}^{t}\int_{B_{\rho}(x)}\left(\log u-\log v\right)_-
(y,s)\,dy\,ds.\end{aligned}\right.
\end{equation*}
Since $u$ and $v$ are the solutions of 
\eqref{eq-cases-main-cauchy-problem}, by the Green Theorem (\cite{GT}) 
and an approximation argument,
\begin{align}\label{eq-aligned-kato-applying-the-green-second-identity-0}
&\frac{\partial}{\partial s}\int_{B_{\rho}(x)}\left( u-v\right)_+(y,s)
(\rho^2-|x-y|^2)\,dy\nonumber\\
=&\int_{B_{\rho}(x)\cap\{u>v\}}
\La\left(\log u-\log v\right)(y,s)(\rho^2-|x-y|^2)\,dy\nonumber\\
\leq&\int_{B_{\rho}(x)\cap\{u>v\}}
\left(\log u-\log v\right)(y,s)\La(\rho^2-|x-y|^2)\,dy\nonumber\\
&\qquad -\int_{\partial\left\{B_{\rho}(x)\cap\{u>v\}
\right\}}\left(\log u-\log v\right)(y,s)\frac{\partial}{\partial \nu}
(\rho^2-|x-y|^2)\,d\sigma_y\nonumber\\
\leq&-2N\int_{B_{\rho}(x)}
\left(\log u-\log v\right)_+(y,s)\,dy\nonumber\\
&\qquad +2\rho\int_{\partial B_{\rho}(x)}\left(\log u-\log v\right)_+
(y,s)\,d\sigma_y\qquad\qquad\forall 0<s<T
\end{align}
where $\frac{\partial}{\partial\nu}$ is the derivative with respect 
to the unit outer normal $\nu$ on $\partial\left\{B_{\rho}(x)
\cap\{u>v\}\right\}$. Integrating 
\eqref{eq-aligned-kato-applying-the-green-second-identity-0} with 
respect to $s$ over $(0,\tau)$, we have
\begin{align}\label{eq-aligned-kato-applying-the-green-second-identity-12345}
&\int_{B_{\rho}(x)}\left( u-v\right)_+(y,\tau)(\rho^2-|x-y|^2)\,dy\nonumber\\
&\qquad \qquad \qquad \leq \int_{B_{\rho}(x)}
\left( u_0-v_0\right)_+(\rho^2-|x-y|^2)\,dy\nonumber\\
&\qquad \qquad \qquad \quad -2N\int_{0}^{\tau}\int_{B_{\rho}(x)}
\left(\log u-\log v\right)_+(y,s)\,dy\,ds\nonumber\\
&\qquad \qquad \qquad \quad +2\rho\int_{0}^{\tau}\int_{\partial B_{\rho}(x)}
\left(\log u-\log v\right)_+(y,s)\,d\sigma_y\,ds\quad\forall 0<\tau<T.
\end{align}
Integrating \eqref{eq-aligned-kato-applying-the-green-second-identity-12345} 
with respect to $\tau$ over $(0,t)$,
\begin{align}
\label{eq-aligned-kato-applying-the-green-second-identity-12345678}
&\int_{0}^{t}\int_{B_{\rho}(x)}\left( u-v\right)_+(y,\tau)(\rho^2-|x-y|^2)
\,dyd\tau\nonumber\\
&\qquad \qquad \qquad \leq T\int_{B_{\rho}(x)}\left( u_0-v_0\right)_+
(\rho^2-|x-y|^2)\,dy\nonumber\\
&\qquad \qquad \qquad \quad -2N\int_{0}^{t}\int_{0}^{\tau}\int_{B_{\rho}(x)}
\left(\log u-\log v\right)_+(y,s)\,dy\,ds\,d\tau\nonumber\\
&\qquad \qquad \qquad \quad +2\rho\int_{0}^{t}\int_{0}^{\tau}
\int_{\partial B_{\rho}(x)}\left(\log u-\log v\right)_+(y,s)
\,d\sigma_y\,ds\,d\tau\quad\forall 0<t<T.
\end{align}
Let $0<t_0<T$ and $\delta=T-t_0$. Now we divide the proof into 
two cases depending on whether
\begin{equation}\label{eq-split-case-into-two-1}
\int_{0}^{t_0}\int_{0}^{\tau}\int_{\R^N}\left(\log u-\log v\right)_+(y,s)
\,dydsd\tau<\infty
\end{equation}
or
\begin{equation}\label{eq-split-case-into-two-2}
\int_{0}^{t_0}\int_{0}^{\tau}\int_{\R^N}\left(\log u-\log v\right)_+(y,s)
\,dy\,ds\,d\tau=\infty.
\end{equation}

\noindent \textbf{Case 1}: \eqref{eq-split-case-into-two-1} holds. 

\noindent Then for any $0<\delta'<t_0$,
\begin{equation}
\begin{aligned}
\infty&>\int_{t_0-\delta'}^{t_0}\int_{0}^{\tau}\int_{\R^N}
\left(\log u-\log v\right)_+(y,s)\,dy\,ds\,d\tau\\
&\geq \delta'\int_{0}^{t}\int_{\R^N}\left(\log u-\log v\right)_+(y,s)\,dy\,ds 
\qquad \qquad \forall 0<t<t_0-\delta'.
\end{aligned}
\end{equation}
Hence
\begin{equation}\label{eq--split-case-into-two-1-result-1}
\lim_{\rho\to\infty}\frac{1}{\rho^N}I_1(\rho,t)=0 \qquad \qquad 
\forall 0<t<t_0-\delta'.
\end{equation}
Since $\delta'$ is arbitrary, \eqref{eq--split-case-into-two-1-result-1} 
holds for any $0<t<t_0$.

\noindent \textbf{Case 2}: \eqref{eq-split-case-into-two-2} holds. 

\noindent By the l'Hospital rule,
\begin{equation}\label{eq-result-lhospital-rule-123}
\begin{aligned}
&\lim_{\rho\to\infty}\frac{1}{\rho^N}\int_{0}^{t_0}\int_{0}^{\tau}\int_{B_{\rho}(x)}
\left(\log u-\log v\right)_+\,dy\,ds\,d\tau\\
&\qquad \qquad \qquad =\lim_{\rho\to\infty}\frac{1}{N\rho^{N-1}}\int_{0}^{t_0}
\int_{0}^{\tau}\int_{\partial B_{\rho}(x)}\left(\log u-\log v\right)_+
\,d\sigma_y\,ds\,d\tau.
\end{aligned}
\end{equation}
Let $r_1=|x|+k\delta^{-\frac{1}{N-2}}$. By 
\eqref{eq-aligned-some-inequality-of-log-1}, 
\eqref{eq-condition-result-of-B--1}, 
Lemma \ref{lem-preliminary-esti-first-1}, and the H\"older inequality,
for any $\rho>|x|+k\delta^{-\frac{1}{N-2}}$ and $0<t\leq T-\delta$,  we have
\begin{align}\label{eq-aligned-bounded-above-of-integral-quantity-467}
&\frac{1}{\rho^N}\int_{0}^{t}\int_{0}^{\tau}\int_{B_{\rho}(x)}
\left(\log u-\log v\right)_+\,dy\,ds\,d\tau\nonumber\\ 
=&\frac{1}{\rho^N}\int_{0}^{t}\int_{0}^{\tau}\int_{B_{\rho}(x)\cap B_{r_1}(0)}
\left(\log u-\log v\right)_+\,dy\,ds\,d\tau\nonumber\\
&\qquad +\frac{1}{\rho^N}\int_{0}^{t}\int_{0}^{\tau}
\int_{B_{\rho}(x)\setminus B_{r_1}(0)}\left(\log u-\log v\right)_+
\,dy\,ds\,d\tau\nonumber\\ 
\le&\frac{C_1}{\rho^N}\int_{0}^{t}\int_{0}^{\tau}\int_{B_{r_1}(0)}
\left(u-v\right)^{\frac{1}{2}}_+(y,s)\,dy\,ds\,d\tau\nonumber\\
&\qquad+\frac{C_1}{\rho^N}\int_{0}^{t}
\int_{0}^{\tau}\frac{\rho}{(T-s)^{\frac{1}{2}}}\left(\int_{B_{\rho}(x)}
(u-v)^{\frac{1}{2}}_+(y,s)\,dy\right)\,ds\,d\tau\nonumber\\
\le&\frac{C'}{\rho^{N}}\int_{0}^{t}\int_{0}^{\tau}\left(
\int_{B_{r_1}(0)}(u-v)_+(y,s)\,dy\right)^{\frac{1}{2}}\,ds\,d\tau
+C'T^{\frac{3}{2}}\left(
\frac{\|f\|^{\frac{1}{2}}_{L^1(\R^N)}}{\rho^{\frac{N}{2}-1}}+\sqrt{T}\right)
\end{align}
for some constant $C_1>0$, $C'>0$, depending on $\delta$ and $k$. By 
\eqref{eq-aligned-bounded-above-of-integral-quantity-467} 
the limit in \eqref{eq-result-lhospital-rule-123} is finite. 
Since $u_0-v_0\in L^1(\R^N)$,
\begin{equation}\label{eq-finiteness-of-inital-value-over-rho-to-n-1}
\lim_{\rho\to\infty}\frac{1}{\rho^{N-2}}\int_{B_{\rho}(x)}
\left(u_0-v_0\right)_+(y)\,dy=0.
\end{equation}
Dividing \eqref{eq-aligned-kato-applying-the-green-second-identity-12345678} 
by $\rho^N$ and letting $t=t_0$ and $\rho\to\infty$ as $i\to\infty$, by 
\eqref{eq-result-lhospital-rule-123} and 
\eqref{eq-finiteness-of-inital-value-over-rho-to-n-1},
\begin{equation}\label{eq-important-estimate-for-lemma-2-1-1}
\lim_{i\to\infty}\frac{1}{\rho^N}\int_{0}^{t_0}\int_{B_{\rho}(x)}
\left(u-v\right)_+(y,\tau)\left(\rho^2-|x-y|^2\right)\,dyd\tau=0.
\end{equation}
Let $0<\epsilon<1/2$. Since
\begin{equation*}
\rho^2\leq \frac{\rho^2-|x-y|^2}{1-(1-\epsilon)^2} 
\qquad\forall y\in B_{(1-\epsilon)\rho}(x),
\end{equation*}
by \eqref{eq-aligned-some-inequality-of-log-1}, 
\eqref{eq-condition-result-of-B--1}, 
Lemma \ref{lem-preliminary-esti-first-1}, and the 
H\"older inequality, for any $\rho>|x|$ we have
\begin{align}\label{eq-important-estimate-for-lemma-2-1-1-1-1}
&\int_{0}^{t_0}\int_{B_{(1-\epsilon)\rho}(x)}\left(\log u-\log v\right)_+
(y,s)\,dy\,ds\nonumber\\ 
=&\int_{0}^{t_0}\int_{B_{(1-\epsilon)\rho}(x)\cap B_{r_1}(0)}
\left(\log u-\log v\right)_+(y,s)\,dy\,ds\nonumber\\ 
&\qquad +\int_{0}^{t_0}\int_{B_{(1-\epsilon)\rho}(x)\setminus  B_{r_1}(0)}
(\log u-\log v)_+(y,s)\,dy\,ds\nonumber\\ 
\le&C_2\int_{0}^{t_0}\int_{B_{r_1}(0)}(u-v)^{\frac{1}{2}}_+(y,s)\,dy\,ds
+C_2\int_{0}^{t_0}\frac{1}{(T-s)^{\frac{1}{2}}}
\left(\int_{B_{(1-\epsilon)\rho}(x)}\rho\cdot\left(u-v\right)^{\frac{1}{2}}_+
\,dy\right)\,ds\nonumber\\
\le&C_3T+\frac{C_3\rho^{\frac{N}{2}}\sqrt{T}}{\delta^{\frac{1}{2}}
\left(1-(1-\epsilon)^2\right)^{\frac{1}{2}}}\left(\int_0^{t_0}
\int_{B_{\rho}(x)}\left(u-v\right)_+(y,s)(\rho^2-|x-y|^2)\,dy\,ds
\right)^{\frac{1}{2}}
\end{align}
for some constants $C_2>0$, $C_3>0$. Thus by 
\eqref{eq-important-estimate-for-lemma-2-1-1} and 
\eqref{eq-important-estimate-for-lemma-2-1-1-1-1},
\begin{equation}\label{eq-main-estimate-of-L-1-contraction-001}
\lim_{\rho\to\infty}\frac{1}{\rho^N}\int_{0}^{t_0}\int_{B_{(1-\epsilon)\rho}(x)}
\left(\log u-\log v\right)_+(y,s)\,dy\,ds=0.
\end{equation}
Now for any $y\in B_{\rho}(x)\setminus B_{(1-\epsilon)\rho}(x)$ and 
$\rho>2(|x|+r_1)$ we have
$$
\frac{3}{2}\rho\geq |x|+\rho\ge |y|\ge |x-y|-|x|\ge (1-\3)\rho-|x|\ge r_1.
$$
Hence by \eqref{eq-aligned-some-inequality-of-log-1}, 
\eqref{eq-condition-result-of-B--1}, 
Lemma \ref{lem-preliminary-esti-first-1}, and the 
H\"older inequality, for any $\rho>2(|x|+r_1)$,
\begin{equation*}
\begin{aligned}
&\int_{0}^{t_0}\int_{B_{\rho}(x)\setminus B_{(1-\epsilon)\rho}(x)}
\left(\log u-\log v\right)_+(y,s)\,dy\,ds\\ 
&\leq C\int_{0}^{t_0}\int_{B_{\rho}(x)\setminus B_{(1-\epsilon)\rho}(x)}v^{-\frac{1}{2}}
\left(u-v\right)^{\frac{1}{2}}_+\,dy\,ds\\
&\leq C\rho\int_{0}^{t_0}\frac{1}{(T-s)^{\frac{1}{2}}}
\int_{B_{\rho}(x)\setminus B_{(1-\epsilon)\rho}(x)}\left(u-v\right)^{\frac{1}{2}}_+\,dy\,ds\\
& \leq C'\left(1-(1-\epsilon)^{N}\right)^{\frac{1}{2}}\sqrt{T}
\rho^{\frac{N}{2}+1}\left(\|f\|_{L^1(\R^N)}+\rho^{N-2}T\right)^{\frac{1}{2}}
\end{aligned}
\end{equation*}
for some constants $C>0$, $C'>0$. Hence
\begin{equation}\label{eq-main-estimate-of-L-1-contraction-002}
\limsup_{\rho\to\infty}\frac{1}{\rho^N}\int_{0}^{t_0}
\int_{B_{\rho}(x)\setminus B_{(1-\epsilon)\rho}(x)}\left(\log u-\log v\right)_+(y,s)\,dyds
\leq C'\left(1-(1-\epsilon)^{N}\right)^{\frac{1}{2}}T.
\end{equation}
By \eqref{eq-main-estimate-of-L-1-contraction-001} and 
\eqref{eq-main-estimate-of-L-1-contraction-002},
\begin{equation}\label{eq-main-estimate-of-L-1-contraction-0021}
\limsup_{\rho\to\infty}\frac{1}{\rho^N}I_1(\rho,t_0) 
\leq C'\left(1-(1-\epsilon)^{N}\right)^{\frac{1}{2}}T.
\end{equation}
Since $0<\epsilon<1/2$ is arbitrary, letting $\epsilon\to 0$ in
\eqref{eq-main-estimate-of-L-1-contraction-0021} we get that
\begin{equation}\label{eq-final-two-of-I-1-0972}
\lim_{\rho\to\infty}\frac{1}{\rho^N}I_1(\rho,t)=0
\end{equation}
holds for any $0<t<t_0$. By Case 1 and  Case 2, 
\eqref{eq-final-two-of-I-1-0972} holds for any $0<t<t_0$. Since 
$0<t_0<T$ is arbitrary, \eqref{eq-final-two-of-I-1-0972} holds 
for any $0<t<T$. Similarly, 
\begin{equation*}
\lim_{\rho\to\infty}\frac{1}{\rho^N}I_2(\rho,t)=0\quad\forall
0<t<T
\end{equation*}
and \eqref{eq-lim-of-w-as-rho-to-infty} follows. Letting 
$\rho\to\infty$ in \eqref{eq-upperbound-newto-Z-1},
by \eqref{eq-lim-of-w-as-rho-to-infty},
\begin{equation}\label{eq-upperbound-w-123}
w(x,t)\leq Z(x) \qquad \forall x\in \R^N, 0<t<T.
\end{equation}
By \eqref{eq-upperbound-w-123}, we get that 
\eqref{eq-upperbound-newto-Z-A} holds for any $f\in L^1(\R^N)
\cap L^{\infty}(\R^N)$.

For general $f\in L^1(\R^N)$. Let $\vp\in C^{\infty}_0
(\R^N)$ be such that $0\leq \vp\leq 1$ and $\int_{\R^N}\vp=1$. Let
\begin{equation*} 
\vp_{\epsilon}(y)=\epsilon^{-N}\vp\left(\frac{y}{\epsilon}\right)
\end{equation*}
and
\begin{equation*}
g_{\epsilon}(x)=g\ast\vp_{\epsilon}(x)=\int_{\R^N}g(x-y)\vp_{\epsilon}(y)
\,dy
\end{equation*}
for any $0<\3<1$ and $g\in L^1(\R^N)$. Then by 
\eqref{eq-in-the-sense-of-distributions},
\begin{equation*}
\La\left(w_{\epsilon}-Z_{\epsilon}\right)\geq 0 \qquad \mbox{in $\R^N$}
\quad\forall 0<t<T.
\end{equation*}
Hence
\begin{equation*}
w_{\epsilon}(x,t)\leq Z_{\epsilon}(x)+\frac{N}{\omega_N\rho^N}
\int_{B_{\rho}(x)}w_{\epsilon}(y,t)\,dy \qquad \forall\rho>0, x\in\R^N, 
0<t<T.
\end{equation*}
Therefore
\begin{equation*}
\int_{R\leq|x|\leq 2R}w_{\epsilon}(x,t)\,dx
\leq \int_{R\leq|x|\leq 2R}Z_{\epsilon}(x)\,dx
+\int_{R\leq|x|\leq 2R}\left(\frac{N}{\omega_N\rho^N}
\int_{B_{\rho}(x)}w_{\epsilon}(y,t)\,dy\right)\,dx\quad\forall R>0.
\end{equation*}
Letting $\epsilon\to 0$,
\begin{align}\label{eq-aligned-letting-epsilon-to-zero-c}
\int_{R\leq|x|\leq 2R}w(x,t)\,dx&\le\int_{R\leq|x|\leq 2R}Z(x)\,dx
+\int_{R\leq|x|\leq 2R}\left(\frac{N}{\omega_N\rho^N}\int_{B_{\rho}(x)}
w(y,t)\,dy\right)\,dx\nonumber\\
&\leq \int_{R\leq|x|\leq 2R}Z(x)\,dx+\int_{R\leq|x|\leq 2R}
\left(\frac{N}{\omega_N\rho^N}\int_{B_{\rho+2R}(0)}w(y,t)\,dy\right)\,dx.
\end{align}
By \eqref{eq-lim-of-w-as-rho-to-infty},
\begin{equation*}
\lim_{\rho\to\infty}\frac{1}{\rho^N}\int_{B_{\rho+2R}(0)}w(y,t)\,dy=0.
\end{equation*}
Hence by letting $\rho\to\infty$ in 
\eqref{eq-aligned-letting-epsilon-to-zero-c}, 
\eqref{eq-upperbound-newto-Z-A} follows. Let $\eta_R$ be as in the proof of 
Lemma~\ref{lem-preliminary-esti-first-1}. By \eqref{eq-E-by-H-1924},
\begin{equation}\label{eq-estimate-J-R-004}
\begin{aligned}
\int_{\R^N}|u-v|(\cdot,t)\eta_R\,dx 
&\leq \int_{\R^N}|f|\,dx+\int_{0}^{t}\int_{R\le |x|\le 2R}
|\log u-\log v||\La \eta_R|\,dxds\\
&\leq \|f\|_{L^1(\R^N)}+\frac{C}{R^2}\int_{R\le |x|\le 2R}w(x)\,dx.
\end{aligned}
\end{equation}
By \eqref{eq-upperbound-newto-Z-A},
\begin{equation}\label{eq-aligned-last-estimate-of-lemma-2-2}
\begin{aligned}
\int_{R\leq|x|\leq 2R}w(x,t)\,dx &\leq \frac{1}{N(N-2)\omega_N}
\int_{R\leq|x|\leq 2R}\left(\int_{\R^N}\frac{|f(y)|}{|x-y|^{N-2}}\,dy\right)\,dx\\
&\leq C\int_{\R^N}|f(y)|
\left(\int_{R\leq|x|\leq 2R}\frac{dx}{|x-y|^{N-2}}\right)\,dy\\
&\leq C\int_{\R^N}|f(y)|J_R(y)\,dy.
\end{aligned}
\end{equation}
where $J_R(y)=\int_{R\leq|x|\leq 2R}\frac{dx}{|x-y|^{N-2}}$. Let 
$R\leq |x|\leq 2R$. Then for $|y|\leq \frac{R}{2}$ we have
$|x-y|\geq |x|/2$. Hence
\begin{equation}\label{eq-estimate-J-R-001}
J_R(y)\leq \int_{R\leq|x|\leq 2R}
\frac{dx}{\left(\frac{|x|}{2}\right)^{N-2}}\leq CR^2\qquad 
\forall |y|\leq \frac{R}{2}.
\end{equation}
For $|y|\geq 4R$, we have $|x-y|\geq |y|/2\geq 2R$. Thus
\begin{equation}\label{eq-estimate-J-R-002}
J_R(y)\leq \int_{R\leq|x|\leq 2R}\frac{dx}{\left(2R\right)^{N-2}}
\leq CR^2\qquad \forall |y|\leq 4R.
\end{equation}
Finally for $\frac{R}{2}<|y|<4R$, we have $|x-y|<6R$. Therefore
\begin{equation}\label{eq-estimate-J-R-003}
J_R(y)\leq \int_{|x-y|<6R}\frac{dx}{\left|x-y\right|^{N-2}}
\leq CR^2\qquad \forall \frac{R}{2}<|y|<4R.
\end{equation}
By \eqref{eq-aligned-last-estimate-of-lemma-2-2}, 
\eqref{eq-estimate-J-R-001}, \eqref{eq-estimate-J-R-002} 
and \eqref{eq-estimate-J-R-003},
\begin{equation}\label{eq-estimate-J-R-005}
\int_{R\leq|x|\leq 2R}w(x,t)\,dx\leq C'R^2\|f\|_{L^1}\quad\forall
0<t<T
\end{equation}
for some constant $C'>0$. By \eqref{eq-estimate-J-R-004} and 
\eqref{eq-estimate-J-R-005},
\begin{equation*}
\int_{\R^N}|u-v|(x,t)\eta_{R}(x)\,dx \leq  C\|f\|_{L^1(\R^N)} 
\qquad \forall R>0, 0<t<T
\end{equation*} 
for some constant $C>0$. Letting $R\to\infty$, we get
\begin{equation*}
\int_{\R^N}|u-v|(x,t)\,dx\le C\|f\|_{L^1(\R^N)}\quad\forall 0<t<T
\end{equation*}
and  the lemma follows.
\end{proof}

By an argument similar to the proof of Corollary 2.2 
of \cite{DS1} but with Lemma \ref{lem-before-L-1-contraction-46} 
replacing Lemma 2.1 of \cite{DS1} in the proof, we have the 
following $L^1$-contraction principle for the solutions of 
\eqref{eq-cases-main-cauchy-problem} that are bounded below by 
some Barenblatt solution $B$.

\begin{lem}\label{eq-L-1-contraction-principle}
Let $u$, $v$ be two solutions of \eqref{eq-cases-main-cauchy-problem} 
with initial values $u_0$, $v_0$ respectively and $f=u_0-v_0\in 
L^1(\R^N)$. Assume in addition that $u$, $v\geq B$, for some 
Barenblatt solution $B=B_k$ given by 
\eqref{eq-self-similar-barenblett-solution}. Then
\begin{equation*}
\int_{\R^N}|u(\cdot,t)-v(\cdot,t)|\,dx \leq \int_{\R^N}|u_0-v_0|\,dx, 
\qquad \forall t\in[0,T).
\end{equation*}
\end{lem}

As a consequence of Lemma \ref{eq-L-1-contraction-principle} 
we have the following result concerning the rescaling solutions 
$\4{u}$ and $\4{v}$ of solutions $u$ and $v$ of 
\eqref{eq-cases-main-cauchy-problem} .

\begin{cor}\label{cor-tilde-u-L-1-contraction-987}
Let $u$, $v$, $u_0$, $v_0$, be as in Lemma \ref{eq-L-1-contraction-principle}. 
If $u_0-v_0\in L^1(\R^N)$, then 
\begin{equation*}
\int_{\R^N}\left|\4{u}(x,s)-\4{v}(x,s)\right|\,dx
\leq \int_{\R^N}|u_0-v_0|\,dx, \qquad \forall s>-\log T.
\end{equation*}
\end{cor}

\section{The integrable case $(N=3)$}\label{section-The-integrable-case-(N=3)}
\setcounter{equation}{0}
\setcounter{thm}{0}

This section will be devoted to the proof of 
Theorem~\ref{lem-main-theorem-with-finite-integral-MS}. Note that when
$N=3$, the difference of two solutions $u$, $v$, satisfying 
\eqref{eq-initial-u-0-trapped-by-B} is integrable. 
We will use a modification of the technique of \cite{Hs1} to prove 
Theorem \ref{lem-main-theorem-with-finite-integral-MS}. We begin 
this section with the following technical lemma, 
which constitutes the main step in the proof of 
Theorem \ref{lem-main-theorem-with-finite-integral-MS}. 

\begin{lem}\label{lem-L-1-integrabilities-of-p-456}
Let $N\ge 3$, $s_0>0$, $0\leq f\in L^1(\R^N)\cap L^{\infty}(\R^N)$ and 
$0\leq g, \hat{g}\in C(\R^N\times(0,s_0])\cap L^{1}(\R^N\times[0,s_0])$ 
such that $0\leq g\leq \hat{g}$ on $\R^N\times(0,s_0)$. Let $\4{a}(x,s)
\in C^{\infty}(\R^N\times(0,s_0])$ satisfy
\begin{equation}\label{eq-condition-of-Lemma-4-first-437}
C_1(1+|x|^2)\le\4{a}(x,s) \le C_2(1+|x|^2)\qquad 
\forall x\in\R^N,0\le s\le s_0, 
\end{equation}
for some constants $C_1>0$, $C_2>0$. For any $R>1$, let $p_R(x,s)$ 
be a solution of 
\begin{equation}\label{eq-case-aligned-problem-for-q-in-B-r}
\begin{cases}
\begin{aligned}
p_{s}(x,s)=&\La\left(\4{a}(x,s)p(x,s)\right)+\frac{1}{N-2}
\mbox{div}\left(x\cdot p(x,s)\right) \qquad 
\mbox{in }B_R(0)\times (0,s_0)\\
p(x,s)=&g(x,s)\qquad \qquad \qquad \qquad\qquad \qquad \qquad
\qquad\quad \mbox{on }\partial B_R(0)\times(0,s_0)\\
p(x,0)=&f(x) \qquad \qquad \qquad\qquad\qquad\qquad \qquad \qquad 
\quad\,\, \mbox{ in }B_R(0)
\end{aligned}
\end{cases}
\end{equation}
Then there exists a sequence of positive numbers 
$\{R_i\}_{i=1}^{\infty}$, $R_i\to\infty$ as $i\to\infty$, 
depending on $\hat{g}$ and independent of $g$ such that $p_{R_i}$ 
converges uniformly on every compact subsets of $\R^N\times(0,s_0]$ as 
$i\to\infty$ to a solution $p$ of 
\begin{equation}\label{eq-difference-between-tilde-u-and-tilde-v}
q_s=\La\left(\4{a}(x,s)q\right)+\frac{1}{N-2}\mbox{div}(x\cdot q)
\end{equation}
in $\R^N\times(0,s_0]$ which satisfies
\begin{equation}\label{eq-moreover-properties-of-q}
\int_{\R^N}p(x,s)\,dx\leq \int_{\R^N}f\,dx \qquad \forall 0<s\leq s_0.
\end{equation}
\end{lem}
\begin{proof}
Since $\hat{g}\in L^1(\R^N\times[0,s_0])$,
\begin{equation}\label{eq-convergence-of-h-using-L-1-integrability-1}
\int_{\frac{i}{2}}^{i}\int_0^{s_0}\int_{|x|=R}\hat{g}(y,s)\,d\sigma_R\,ds\,dR 
\to 0 \qquad \mbox{as $i\to\infty$}
\end{equation}
where $d\sigma_R$ is the surface measure on $\partial B_R(0)$. For each 
$i\in\N$, there exists $R_i\in[i/2,i]$ such that
\begin{equation}
\int_0^{s_0}\int_{|x|=R_i}\hat{g}(y,s)\,d\sigma_{R_i}ds
=\min_{\frac{i}{2}\leq R\leq i}\left\{\int_{0}^{s_0}
\int_{|x|=R}\hat{g}(y,s)\,d\sigma_Rds\right\}.
\end{equation}
Then by \eqref{eq-convergence-of-h-using-L-1-integrability-1},
\begin{align}
&\frac{i}{2}\int_0^{s_0}\int_{|x|=R_i}\hat{g}(y,s)\,d\sigma_{R_i}ds 
\to 0 \qquad \mbox{as }i\to\infty\notag\\
\Rightarrow\quad&R_i\int_0^{s_0}\int_{|x|=R_i}\hat{g}(y,s)\,d\sigma_{R_i}ds\to 0 
\qquad\mbox{as }i\to\infty.
\label{eq-convergence-of-h-using-L-1-integrability-2}
\end{align}
By choosing a subsequence if necessary we may assume without loss of 
generality that $R_{i+1}>R_i$ for any $i\in\N$.
By \eqref{eq-condition-of-Lemma-4-first-437} and the Schauder 
estimates for parabolic equations \cite{LSU}, the sequence 
$\{p_{R_i}\}_{i=1}^{\infty}$ is equi-H\"older continuous in $C^{2,1}$ 
on every compact subsets of $\R^N\times(0,s_0]$. Hence by the Ascoli 
Theorem and a diagonalization argument there exists a subsequence, 
which we will still denote by $\{p_{R_i}\}_{i=1}^{\infty}$, that converges 
uniformly on every compact subsets of $\R^N\times(0,s_0]$ to a 
solution $p$ of \eqref{eq-difference-between-tilde-u-and-tilde-v}
in $\R^N\times(0,s_0]$ as $i\to\infty$.

It remains to prove \eqref{eq-moreover-properties-of-q}. We
fix $s_1\in(0,s_0]$ and define the operator $L$ by
\begin{equation*}
L[\psi]=\psi_s+\4{a}\La\psi-\frac{1}{N-2}x\cdot\nabla\psi.
\end{equation*}
For any $R>1$ and $h\in C_0^{\infty}(B_R(0))$, $0\leq h\leq 1$ on $B_R(0)$, 
such that
\begin{equation}\label{eq-terminal-values-in-first-Lemma-of-4th-section-5}
h(x)=0 \quad \mbox{on $B_R(0)\bs B_{\frac{R}{2}}(0)$},
\end{equation}
let $\psi_R(x,s)$ be the solution of 
\begin{equation}\label{psi-eqn}
\begin{cases}
\begin{aligned}
L[\psi]&=0 \qquad \mbox{ in $B_R(0)\times(0,s_1)$}\\
\psi(x,s)&=0 \qquad \mbox{ on $\partial B_R(0)\times(0,s_1)$}\\
\psi(x,s_1)&=h(x) \quad\mbox{in }B_R(0).
\end{aligned}
\end{cases}
\end{equation} 
By the maximum principle $0\leq\psi_R\le 1$ in $B_R(0)\times(0,s_1)$.
Let 
$$
H_k(x)=\frac{2^{2k}}{2^{2k}-1}\left(1-\frac{|x|^{2k}}{R^{2k}}\right)
$$ 
for some $k>0$ to be determined later. By direct computation for any 
$k\in\N$
\begin{equation}\label{eq-operator-values-in-first-Lemma-of-4th-section-5}
\begin{aligned}
L\left[H_k(x)\right]\leq -\left(\frac{2^{2k}}{2^{2k}-1}\right)
\frac{2k|x|^{2k-2}}{R^{2k}}\left[C_1(N+2k-2)
+\left(C_1(N+2k-2)-\frac{1}{N-2}\right)|x|^2\right]
\end{aligned}
\end{equation}
on $B_R(0)\bs B_{R/2}(0)$ and 
\begin{equation}\label{eq-boundary-condition-of-H-39}
\begin{cases}
\begin{aligned}
H_k(x)\equiv &0 \qquad \qquad \qquad \quad\forall |x|=R\\
H_k(R/2)=&1\geq \psi_R(x,s)\qquad\forall |x|=R/2, 0<s<s_1.
\end{aligned}
\end{cases}
\end{equation}
We now choose $k>\frac{1}{2C_1(N-2)}+1$. Then by 
\eqref{eq-operator-values-in-first-Lemma-of-4th-section-5},
\begin{equation}\label{Hk-eqn}
L\left[H_{k}(x)\right]<0\quad\mbox{ in }B_R(0)\bs B_{\frac{R}{2}}(0). 
\end{equation}
Hence $H_k(x)$ is a super-solution of $L(\xi)=0$ in 
$(B_R(0)\bs B_{R/2}(0))\times(0,s_1)$. 
By \eqref{eq-terminal-values-in-first-Lemma-of-4th-section-5}, 
\eqref{psi-eqn}, \eqref{eq-boundary-condition-of-H-39}, \eqref{Hk-eqn}, 
and the maximum principle in $(B_R(0)\bs B_{R/2}(0))\times(0,s_1)$,
\begin{equation}\label{eq-boundary-condition-of-H-394}
\psi_{R}(x,s) \leq H_{k}(x)=\frac{2^{2k}}{2^{2k}-1}
\left(1-\frac{|x|^{2k}}{R^{2k}}\right)\quad \mbox{ on }
(B_R(0)\bs B_{R/2}(0))\times(0,s_1).
\end{equation}
Then by \eqref{eq-boundary-condition-of-H-39} and 
\eqref{eq-boundary-condition-of-H-394},
\begin{equation*}
\left|\frac{\partial \psi_{R}}{\partial \nu}\right|\leq 
\left|\frac{\partial}{\partial \nu}\left(\frac{2^{2k}}{2^{2k}-1}
\left(1-\frac{|x|^{2k}}{R^{2k}}\right)\right)\right|
\leq \frac{C}{R} \qquad \mbox{on $\partial B_R(0)\times(0,s_1)$}
\end{equation*}
for some constant $C>0$ depending on $k$ where 
$\frac{\partial}{\partial\nu}$ is the derivative with respect 
to the unit outer normal $\nu$ on the boundary $\partial B_R(0)$.

Multiplying \eqref{eq-case-aligned-problem-for-q-in-B-r}
by $\psi_R$ and integrating over $B_R(0)$, by integration by parts,
\eqref{eq-condition-of-Lemma-4-first-437}, and \eqref{psi-eqn}, we get 
\begin{equation}\label{eq-alinged-time-derivatives-of-q-098}
\begin{aligned}
\frac{\partial}{\partial s}\left[\int_{|x|\le R}p_R\,\psi_R\,dx\right]&
=\int_{|x|\le R}\left[\psi_{R,s}+\4{a}\La\psi_R-\frac{1}{N-2}x
\cdot\nabla\psi_R\right]p_R\,dx-\int_{|x|=R}\4{a}\,g\,
\frac{\partial\psi_{R}}{\partial \nu}\,d\sigma_R\\
&=-\int_{|x|=R}\4{a}\,g\,
\frac{\partial\psi_{R}}{\partial \nu}\,d\sigma_R\\
&\leq CR\int_{|x|=R}\hat{g}\,d\sigma_R\qquad\qquad\quad
\forall 0<s<s_1, \quad R>1.
\end{aligned}
\end{equation}
Hence
\begin{equation}\label{eq-I-by-H-55}
\int_{|x|\le R}p_R(x,s_1)h(x)\,dx\leq \int_{|x|\le R}f(x)\psi_R(x,0)\,dx
+CR\int_{0}^{s_0}\int_{|x|=R}\hat{g}\,d\sigma_Rds.
\end{equation}
We now choose $h(x)=\eta_{R/4}(x)$ where $\eta_{R/4}(x)$ is as in the proof of 
Lemma~\ref{lem-preliminary-esti-first-1}. By the maximum principle
$p_R\ge 0$ in $B_R\times (0,\infty)$. Then putting $R=R_i$ in 
\eqref{eq-I-by-H-55} and letting $i\to\infty$, 
by \eqref{eq-convergence-of-h-using-L-1-integrability-2},
\begin{equation*}
\int_{\R^N}p(x,s_1)\,dx\leq \int_{\R^N}f\,dx.
\end{equation*}
Since $0<s_1\leq s_0$ is arbitrary, \eqref{eq-moreover-properties-of-q} 
follows. 
\end{proof}

\begin{lem}\label{lem-first-lemma-for-aymptotic-case}
Let $N\ge 3$. Let $u$, $v$, be two solutions of 
\eqref{eq-cases-main-cauchy-problem} with initial values $u_0$, $v_0$, 
satisfying \eqref{eq-initial-u-0-trapped-by-barenblat} for some constants 
$k_1>k_2>0$ and let $\4{u}$, $\4{v}$, be given by 
\eqref{eq-rescaled-function} with $u=u$, $v$, respectively. Let 
$\4{u}_0(x)=\4{u}(x,-\log T)$ and $\4{v}_0(x)=\4{v}(x,-\log T)$. 
Suppose $u$, $v$, satisfy \eqref{eq-initial-u-0-trapped-by-B} and
\begin{equation*}
\min(\|(\4{u}_0-\4{v}_0)_+\|_{L^{\infty}(\R^N)},\|(\4{u}_0-\4{v}_0)_-\|_{L^{\infty}(\R^N)})>0.
\end{equation*}
Then
\begin{equation}\label{eq-tilde-u-tilde-v-strong-decreasing-0012}
\left\|(\4{u}-\4{v})(\cdot,s)\right\|_{L^1(\R^N)}
<\left\|\4{u}_0-\4{v}_0\right\|_{L^1(\R^N)}\quad\forall s>-\log T.
\end{equation}
\end{lem}
\begin{proof}
We will use a modification of the proof of Lemma 2.1 of \cite{Hs1} 
to prove the lemma (cf. Lemma 3.1 of \cite{DS1}).  Let $q=\4{u}-\4{v}$. 
Then $q$ satisfies \eqref{eq-difference-between-tilde-u-and-tilde-v} 
in $\R^N\times (-\log T,\infty)$ with
\begin{equation}\label{eq-mean-value-thm-coeffi-a-1}
\4{a}(x,s)=\int_0^1\frac{d\theta}{\theta\4{u}+(1-\theta)\4{v}}.
\end{equation}
Since both $\4{u}$ and $\4{v}$ satisfy 
\eqref{eq-initial-tilde-u-trapped-by-barenblat}, $\4{a}(x,s)$ satisfies 
the growth estimate
\begin{equation}\label{eq-mean-value-thm-coeffi-a-bound-lower-upper-section-3}
\frac{k_2+|x|^2}{2(N-2)}\leq \4{a}(x,s) \leq \frac{k_1+|x|^2}{2(N-2)}.
\end{equation}
Hence \eqref{eq-difference-between-tilde-u-and-tilde-v} is 
uniformly parabolic on any compact subset of 
$\R^N\times(-\log T,\infty)$.

For any $R>0$, by standard parabolic theory there exist solutions 
$q_1^R$, $q_2^R$ of \eqref{eq-difference-between-tilde-u-and-tilde-v} in 
$Q_R=B_R(0)\times(-\log T,\infty)$ with initial values 
$q_+(\cdot,-\log T)$, $q_-(\cdot,-\log T)$ and boundary value 
$q_+$, $q_-$ on $\partial B_R(0)\times(-\log T,\infty)$, respectively. 
Notice that $q_1^R-q_2^R$ is a solution of 
\eqref{eq-difference-between-tilde-u-and-tilde-v} in 
$Q_R$ with initial value $q(\cdot,-\log T)$ and boundary values 
$q$. By the maximum principle $q=q^R_1-q^R_2$ on $Q_R$. Similarly 
there are solutions $\overline{q}_1^R$, $\overline{q}_2^R$ of 
\eqref{eq-difference-between-tilde-u-and-tilde-v} in $Q_R$ 
with initial values $q_+(\cdot,-\log T)$, $q_-(\cdot,-\log T)$ and zero 
lateral boundary value. By the maximum principle 
\begin{equation}\label{q^R-comparsion}
\begin{cases}
\begin{aligned}
0\le\2{q}^R_1&\leq q^R_1 \quad\mbox{ and }\quad 0\le\2{q}^R_2\le q^R_2
\qquad\mbox{ in }Q_R\\
\2{q}_1^R&\leq\2{q}_1^{R'} \quad\mbox{ and }\quad \2{q}_2^R
\leq \2{q}_2^{R'}\qquad\quad\,\,\mbox{ in }Q_R\quad \forall R'\geq R>0.
\end{aligned}
\end{cases}
\end{equation}
Since both $\4{u}$ and $\4{v}$ satisfy 
\eqref{eq-initial-tilde-u-trapped-by-barenblat}, 
\begin{equation}\label{q-bound}
|q|\le\4{B}_{k_2}-\4{B}_{k_1}\quad\mbox{ in }\R^N\times (-\log T,\infty).
\end{equation}
By \eqref{q^R-comparsion} and \eqref{q-bound}
the families of solutions $\overline{q}_1^{R}(x,s)$ and 
$\overline{q}_2^{R}(x,s)$ are monotone increasing in $R$ and 
uniformly bounded above by $\4{B}_{k_2}-\4{B}_{k_1}$, which implies
\begin{equation*}
\overline{q}_1=\lim_{R\to\infty}\overline{q}_1^R \qquad \mbox{and} 
\qquad \overline{q}_2=\lim_{R\to\infty}\overline{q}_2^R
\end{equation*}
exist and are both solutions of 
\eqref{eq-difference-between-tilde-u-and-tilde-v} in $\R^N\times(0,\infty)$.

Let $\eta_{R'}\in C^{\infty}_0(\R^N)$ be as in the proof of 
Lemma~\ref{lem-preliminary-esti-first-1}. 
By Lemma~\ref{lem-before-L-1-contraction-46} and the same 
computation as the proof of Lemma 2.1 of \cite{Hs1},
\begin{equation*}
\begin{aligned}
&\int_{\R^N}|q(x,s)|\eta_{R'}(x)\,dx
-\int_{\R^N}|\4{u}_0-\4{v}_0|(x)\eta_{R'}(x)\,dx\\
&\qquad \qquad =\int_{-\log T}^{s}\int_{|x|\le R}\left(\4{a}(x,\tau)
q^R_1(x,\tau)\La\eta_{R'}-\frac{1}{N-2}q^R_1(x,\tau)x
\cdot\nabla\eta_{R'}\right)\,dx\,d\tau\\
&\qquad \qquad \qquad +\int_{-\log T}^{s}\int_{|x|\le R}\left(
\4{a}(x,\tau)q^R_2(x,\tau)\La \eta_{R'}-\frac{1}{N-2}q^R_2
(x,\tau)x\cdot\nabla\eta_{R'}\right)\,dx\,d\tau\\
&\qquad \qquad \qquad -2\int_{|x|\le R}\min(q_1^R(x,s),q_2^R(x,s))
\eta_{R'}(x)\,dx\qquad\quad\forall R\ge 2R'>0, s>-\log T.
\end{aligned}
\end{equation*}
Hence
\begin{equation}\label{eq-aligned-call-W-by-H-048}
\begin{aligned}
&\int_{\R^N}|q(x,s)|\eta_{R'}(x)\,dx
-\int_{\R^N}|\4{u}_0-\4{v}_0|(x)\eta_{R'}(x)\,dx\\
\le&\frac{C}{{R'}^2}\int_{-\log T}^{s}
\int_{R'\le |x|\le 2R'}\4{a}(x,\tau)q^R_1(x,\tau)
\,dxd\tau+C\int_{-\log T}^{s}\int_{R'\leq|x|\leq {2R'}}q^R_1(x,\tau)
\,dx\,d\tau\\
&\qquad +\frac{C}{{R'}^2}\int_{-\log T}^{s}
\int_{R'\le |x|\le 2R'}\4{a}(x,\tau)q^R_2(x,\tau)\,dxd\tau
+C\int_{-\log T}^{s}\int_{R'\le |x|\le 2R'}q^R_2(x,\tau)
\,dx\,d\tau\\
&\qquad -2\int_{|x|\le R_0}\min (\overline{q}_1^R(x,s),\2{q}_2^R(x,s))
\eta_{R'}(x)\,dx\quad\forall R\ge 2R'>0,R\ge R_0>0,s>-\log T.
\end{aligned}
\end{equation}
By Corollary \ref{cor-tilde-u-L-1-contraction-987}, 
\begin{equation}\label{q-L^1}
0\le q_+,q_-\le |q|\in L^1(\R^N\times(-\log T,s))\quad\forall s>-\log T.
\end{equation}
Let $s>-\log T$ be fixed. Then by \eqref{q-L^1} and 
Lemma \ref{lem-L-1-integrabilities-of-p-456} there exists a sequence 
of positive numbers $\{R_i\}_{i=1}^{\infty}$, $R_i\to\infty$ as 
$i\to\infty$, such that $q_1^{R_i}$, $q_2^{R_i}$, converges uniformly 
on every compact subset of $\R^N\times(-\log T,s]$ to some solutions 
$\4{q}_1$, $\4{q}_2$, of \eqref{eq-difference-between-tilde-u-and-tilde-v} 
respectively as $i\to\infty$. Moreover
\begin{equation*}
\begin{cases}
\int_{\R^N}\4{q}_1(x,\tau)\,dx\leq \int_{\R^N}|\4{u}_0-\4{v}_0|\,dx 
\qquad \forall -\log T\le\tau\le s\\
\int_{\R^N}\4{q}_2(x,\tau)\,dx\leq \int_{\R^N}|\4{u}_0-\4{v}_0|\,dx 
\qquad \forall -\log T\le\tau\le s.
\end{cases}
\end{equation*}
Hence 
\begin{equation}\label{eq-L-1-integral-of-tilde-q-pm-657}
\4{q}_1, \4{q}_2\in L^1(\R\times[0,s]).
\end{equation}
Putting $R=R_i$ in \eqref{eq-aligned-call-W-by-H-048} and letting
$i\to\infty$, by \eqref{eq-mean-value-thm-coeffi-a-bound-lower-upper-section-3},
\begin{equation}\label{eq-aligned-call-W-by-H-0489}
\begin{aligned}
&\int_{\R^N}|q(x,s)|\eta_{R'}(x)\,dx-\int_{\R^N}|\4{u}_0-\4{v}_0|(x)
\eta_{R'}(x)\,dx\\
\le&C\int_{-\log T}^{s}\int_{R'\le |x|\le 2R'}
\4{q}_1(x,\tau)\,dx\,d\tau+C\int_{-\log T}^{s}\int_{R'\leq|x|\leq 2R'}
\4{q}_2(x,\tau)\,dx\,d\tau\\
&\qquad -2\int_{|x|\le R_0}\min (\2{q}_1(x,s),\2{q}_2(x,s))\eta_{R'}\,dx
\quad\forall R'\ge 1,R_0>0.
\end{aligned}
\end{equation}
By \eqref{eq-L-1-integral-of-tilde-q-pm-657},
\begin{equation*}
\int_{-\log T}^{s}\int_{R'\le |x|\le 2R'}\4{q}_{j}(x,\tau)
\,dxd\tau\to 0 \qquad \mbox{as $R'\to\infty$}, \,\, j=1,2.
\end{equation*}
Hence letting $R'\to\infty$ in \eqref{eq-aligned-call-W-by-H-0489},
\begin{equation}\label{q-q1-q2}
\int_{\R^N}|q(x,s)|\,dx\le\int_{\R^N}|\4{u}_0-\4{v}_0|(x)\,dx 
-2\int_{|x|\le R_0}\min (\overline{q}_1(x,s),\2{q}_2(x,s))\,dx
\end{equation}
holds for any $R_0>0$, $s>-\log T$. We now choose $R_0>0$ such that
\begin{equation*}
\min(\|(\4{u}_0-\4{v}_0)_+\|_{L^{\infty}(B_{R_0}(0))},
\|(\4{u}_0-\4{v}_0)_-\|_{L^{\infty}(B_{R_0}(0))})>0.
\end{equation*}
Since $\2{q}_1\ge\overline{q}_+^{2R_0}$ and $\2{q}_2\ge\overline{q}_-^{2R_0}$, 
by \eqref{q-q1-q2},
\begin{equation*}
\int_{\R^N}|q(x,s)|\,dx-\int_{\R^N}|\4{u}_0-\4{v}_0|(x)\,dx 
\leq -2\int_{|x|\le R_0}\min (\overline{q}_{+}^{2R_0}(x,s),
\overline{q}_{-}^{2R_0}(x,s))\,dx\quad\forall s>-\log T.
\end{equation*}
Since $\overline{q}_{+}^{2R_0}(x,s)$ and $\overline{q}_{-}^{2R_0}(x,s)$ 
are the solutions of \eqref{eq-difference-between-tilde-u-and-tilde-v} 
in $Q_{2R_0}$ with zero boundary value and initial values 
$q_+(\cdot,-\log T)$, $q_+(\cdot,-\log T)$, respectively, by the 
Green function representation for solutions, for any 
$s>-\log T$, there exists a constant $c(s)$ such that 
\begin{equation*}
\min_{|x|\le R_0}\4{q}_+^{2R_0}\geq c(s)>0\qquad \mbox{and} 
\qquad \min_{|x|\le R_0}\4{q}_-^{2R_0}\geq c(s)>0
\end{equation*} 
and the lemma follows.
\end{proof}

We next note that $\4{B}_k$ given by \eqref{eq-rescaled-Barenblatt-23}
is a stationary solution of \eqref{eq-rescaled-main-equation-of-v}
for any $k>0$. By an argument similar to the proof of Lemma 1 of \cite{OR} 
we have the following lemma.

\begin{lem}(cf. Lemma 1 of \cite{OR})
\label{Osher-and-Ralston-Lemma-0}
Suppose $\|\4{u}(\cdot,s_i)-\4{w}_0\|_{L^1(\R^N)}\to 0$ as 
$i\to\infty$.  If $\4{w}$ is a 
solution of \eqref{eq-rescaled-main-equation-of-v} in 
$\R^N\times[0,\infty)$ with initial value $\4{w}(x,0)=\4{w}_0(x)$, then
\begin{equation*}
\|\4{w}(\cdot,s)-\4{B}_k\|_{L^1(\R^N)}=\|\4{w}_0-\4{B}_k\|_{L^1(\R^N)}
\quad\forall s>0,k>0
\end{equation*}
where $\4{B}_k$ is given by \eqref{eq-rescaled-Barenblatt-23}.
\end{lem}

\begin{proof}[\textbf{Proof of Theorem \ref{lem-main-theorem-with-finite-integral-MS}}]
Since the proof of the case $N=3$ is similar to that of \cite{Hs1} and 
section 3 of \cite{DS1}, we will only sketch the argument here. Let
\begin{equation*}
f(k)=\int_{\R^N}(u_0(x)-B_k(x))\,dx.
\end{equation*}
Then $f(k)$ is a continuous monotone increasing function of $k>0$. 
By \eqref{eq-initial-u-0-trapped-by-barenblat}, $f(k_1)\ge 0\ge f(k_2)$. 
Hence by the intermediate value theorem 
there exists a unique $k_0$ such that $f(k_0)=0$. By 
\eqref{eq-initial-tilde-u-trapped-by-barenblat},
Lemma \ref{lem-first-lemma-for-aymptotic-case}, 
Lemma \ref{Osher-and-Ralston-Lemma-0}, and an argument similar 
to the proof of Theorem 2.3 in \cite{Hs1}, one gets that 
the rescaled function $\4{u}(\cdot,s)$ converges uniformly on 
$\R^3$, and also in $L^1(\R^3)$, to the rescaled Barenblatt solution 
$\4{B}_{k_0}$ as $s\to\infty$.
\end{proof}

\section{The non-integrable case $(N \ge 5)$}
\label{section-The-Non-Integral-Case-I-N-geq-4}
\setcounter{equation}{0}
\setcounter{thm}{0}

In this section we will prove Theorem~\ref{lem-main-theorem-MS}. 
Since the difference of any two solutions 
$u$, $v$ of \eqref{eq-cases-main-cauchy-problem} that satisfies 
\eqref{eq-initial-u-0-trapped-by-B} may not be integrable when $N\ge 4$, 
for any solution $u$ that satisfies \eqref{eq-initial-u-0-trapped-by-B}
we cannot ensure the existence of a constant $k_0>0$ such that  
\eqref{initial-mean=0} holds from the condition 
\eqref{eq-initial-u-0-trapped-by-B} alone. 
Thus we need additional conditions on the initial data to ensure convergence. 
We will assume that $u_0$ also satisfies \eqref{u0-B-in-l1} for some constant 
$k_0>0$ and function $f\in L^1(\R^N)$ in this section.
Unless stated otherwise in this section we will assume that $u$ is a 
solution of \eqref{eq-cases-main-cauchy-problem} which satisfies the
bound \eqref{eq-initial-u-0-trapped-by-B}, $\4{u}$ will 
denote the rescaled solution defined by \eqref{eq-rescaled-function}, 
and $\4{B}_k$ will be the rescaled Barenblatt solution given by 
\eqref{eq-rescaled-Barenblatt-23}. 

We will use a modification of the technique of \cite{DS1} to find the 
asymptotic behaviour of the solution of \eqref{eq-cases-main-cauchy-problem} 
near its extinction time $T$. The following simple convergence result 
will be used in the proof of Theorem~\ref{lem-main-theorem-MS}.

\begin{lem}\label{lem-convergence-on-compact-subset-of-R-N=4}
Let $u_0$ satisfy \eqref{eq-initial-u-0-trapped-by-barenblat} for some 
constants $k_2>k_1>0$ and $u$ be a solution of 
\eqref{eq-cases-main-cauchy-problem} that satisfies
\eqref{eq-initial-u-0-trapped-by-B}. Let $\4{u}$ be given by 
\eqref{eq-rescaled-function}.
Let $\{s_i\}_{i=1}^{\infty}$ be a sequence of positive numbers such that 
$s_i\to\infty$ as $i\to\infty$ and $\4{u}_i(\cdot,s)=\4{u}(\cdot,s_i+s)$. 
Then the sequence $\{\4{u}_i\}_{i=1}^{\infty}$ has a subsequence 
$\{\4{u}_{i_k}\}_{k=1}^{\infty}$ that converges uniformly 
on every compact subsets of $\R^N\times(-\infty,\infty)$ to a solution 
$\4{w}(x,s)$ of \eqref{eq-rescaled-main-equation-of-v} in 
$\R^N\times(-\infty,\infty)$ which satisfies 
\eqref{eq-initial-tilde-u-trapped-by-barenblat} in $\R^N\times
(-\infty,\infty)$ as $k\to\infty$.
\end{lem}
\begin{proof}
Since $\4{u}$ satisfies 
\eqref{eq-initial-tilde-u-trapped-by-barenblat} in $\R^N\times
(-\log T,\infty)$, equation 
\eqref{eq-rescaled-main-equation-of-v} is uniformly parabolic 
on $B_{R}\times\left[-\frac{\log T}{2}-s_i,\infty\right)$, for 
any $R>0$. By the Schauder estimates for parabolic parabolic equation 
\cite{LSU} the sequence $\4{u}_i$ is equi-H\"older continuous in $C^2$
on every compact subsets of $\R^N\times(-\infty,\infty)$. 
Hence by the Arzela-Ascoli theorem and a diagonalization argument
the sequence $\{\4{u}_i\}_{i=1}^{\infty}$ has a convergent 
subsequence $\{\4{u}_{i_k}\}_{k=1}^{\infty}$ that converges uniformly 
in $C^2$ on every compact subsets of $\R^N\times(-\infty,\infty)$ to a solution 
$\4{w}$ of \eqref{eq-rescaled-main-equation-of-v} in 
$\R^N\times(-\infty,\infty)$ which satisfies 
\eqref{eq-initial-tilde-u-trapped-by-barenblat} in 
$\R^N\times(-\infty,\infty)$ as $k\to\infty$.
\end{proof}

\begin{lem}\label{lem-initial-L-tilde-B-to-L-tilde-B}
Let $N\ge 5$ and let $\4{u}$, $\4{v}$, be two solutions of 
\eqref{eq-rescaled-main-equation-of-v} with initial values $\4{u}_0$, 
$\4{v}_0$, respectively which satisfy 
\eqref{eq-initial-tilde-u-trapped-by-barenblat}. Let $\4{B}=\4{B}_{k_2}$. 
Suppose $\4{u}_0-\4{v}_0\in L^1(\4{B}^{\alpha},\R^N)$ with $\alpha
=\frac{N-4}{2}$. Then there exists a constant $C>0$ such that 
\begin{equation}\label{Barenblatt-L1-bound}
\int_{\R^N}|\4{u}-\4{v}|(x,s)\4{B}^{\alpha}(x)\,dx
\leq \int_{\R^N}|\4{u}_0-\4{v}_0|\4{B}^{\alpha}(x)\,dx+Cs\quad\forall 
s>-\log T.
\end{equation}
\end{lem}
\begin{proof}
Let $\eta_R\in C^{\infty}_0(\R^N)$ be as in the proof of 
Lemma \ref{lem-preliminary-esti-first-1} and let $q=\4{u}-\4{v}$. 
By the Kato inequality \cite{K} $q$ satisfies
\begin{equation*}
|q|_{s}\leq \La(|\log\4{u}-\log\4{v}|)+\frac{1}{N-2}\nabla(x\cdot|q|)
\quad\mbox{ in }\R^N\times (-\log T,\infty)
\end{equation*}
in the distribution sense. Then 
\begin{align*}
&\frac{d}{ds}\int_{\R^N}|\4{u}-\4{v}|(x,s)\4{B}^{\alpha}(x)\eta_R(x)\,dx
\nonumber\\
\le&\int_{\R^N}|\log\4{u}-\log\4{v}|(x,s)\left(\4{B}^{\alpha}(x)
\La\eta_R(x)+\eta_R(x)\La\4{B}^{\alpha}(x)+2\nabla\4{B}^{\alpha}(x)
\cdot\nabla\eta_R(x)\right)\,dx\nonumber\\
&\qquad -\frac{1}{N-2}\int_{\R^N}|\4{u}-\4{v}|(x,s)\,x
\cdot\left\{\eta_R(x)\nabla\4{B}^{\alpha}(x)+\4{B}^{\alpha}(x)\nabla
\eta_{R}(x)\right\}\,dx.
\end{align*}
Hence
\begin{align}\label{eq-Kato-inequality-2}
&\frac{d}{ds}\int_{\R^N}|\4{u}-\4{v}|(x,s)\4{B}^{\alpha}(x)\eta_R(x)\,dx\nonumber\\
\le&\int_{\R^N}|\log\4{u}-\log\4{v}|(x,s)
\left(\4{B}^{\alpha}(x)\La\eta_R(x)+2\nabla\4{B}^{\alpha}(x)\cdot
\nabla\eta_R(x)\right)\,dx\nonumber\\
&\qquad -\frac{1}{N-2}\int_{\R^N}|\4{u}-\4{v}|(x,s)
\4{B}^{\alpha}(x)\,x\cdot\nabla\eta_{R}(x)\,dx\nonumber\\
&\qquad +\int_{\R^N}\left\{\4{a}(x,s)\La
\4{B}^{\alpha}(x)-\frac{1}{N-2}x\cdot\nabla\4{B}^{\alpha}(x)\right\}
|\4{u}-\4{v}|(x,s)\eta_R(x)\,dx\nonumber\\
=&I_{1,R}+I_{2,R}+I_{3,R}\qquad\qquad\forall s>-\log T.
\end{align}
where $\4{a}(x,s)$ is given by \eqref{eq-mean-value-thm-coeffi-a-1}. 
By direct computation,
\begin{equation}\label{eq-direct-computation-La-tilde-B-alpha-1}
\La\4{B}^{\alpha}(x)=-\frac{(N-4)(2|x|^2+k_2N)}{(k_2+|x|^2)^2}\4{B}^{\alpha}
<0\quad\mbox{ in }\R^N.
\end{equation}
Since $\4{u}$, $\4{v}$, satisfies 
\eqref{eq-initial-tilde-u-trapped-by-barenblat}, by \eqref{eq-mean-value-thm-coeffi-a-1}
$\4{a}(x,s)$ satisfies 
\begin{equation}\label{eq-mean-value-thm-coeffi-a-bound-lower-upper}
\frac{k_2+|x|^2}{2(N-2)}\le\4{a}(x,s)\le\frac{k_1+|x|^2}{2(N-2)}
\quad\mbox{ in }\R^N\times (-\log T,\infty).
\end{equation}
Then by \eqref{eq-direct-computation-La-tilde-B-alpha-1}
and \eqref{eq-mean-value-thm-coeffi-a-bound-lower-upper},
\begin{align}\label{eq-direct-computation-La-tilde-B-alpha-nabla-tilde-1}
\4{a}(x,s)\La\4{B}^{\alpha}(x)-\frac{1}{N-2}x\cdot\nabla\4{B}^{\alpha}(x)
&\leq \frac{k_2+|x|^2}{2(N-2)}\La\4{B}^{\alpha}-\frac{1}{N-2}x
\cdot\nabla\4{B}^{\alpha}(x)\nonumber\\
&=-\frac{k_2(N-4)N}{2(N-2)(k_2+|x|^2)}\4{B}^{\alpha}(x)<0
\end{align}
in $\R^N\times (-\log T,\infty)$.
Hence 
\begin{equation}\label{I3-bound}
I_{3,R}\le 0.
\end{equation} 
Since $\4{u}$, $\4{v}\geq \4{B}$, 
\begin{align*}\label{eq-aligned-some-inequality-of-log-2}
|\log\4{u}-\log\4{v}|=&\left|\log\left(\frac{\4{u}}{\4{v}}\right)\right|
\le\left\{\begin{aligned}
&C\left|(\4{u}/\4{v})-1\right|\quad\mbox{ if }\4{u}\ge\4{v}\\
&C\left|(\4{v}/\4{u})-1\right|\quad\mbox{ if }\4{v}\ge\4{u}
\end{aligned}\right.\\
\le&C\4{B}^{-1}|\4{u}-\4{v}|
\end{align*}
for some constant $C>0$. Then
\begin{equation}\label{eq-aligned-estimate-first-I-1-1}
\left|I_{1,R}\right| \leq C_1\int_{B_{2R}\bs B_{R}}|\4{u}-\4{v}|(x,s)
\4{B}^{-1}(x)\left|\4{B}^{\alpha}(x)\La\eta_R(x)+2\nabla\4{B}^{\alpha}(x)
\cdot\nabla\eta_R(x)\right|\,dx
\end{equation}
Since
\begin{equation}\label{rescaled-tilde-B-bound}
|\4{B}|\le\frac{C}{R^2},\quad|\4{B}^{-1}|\le CR^2,\quad 
|\nabla\4{B}|\leq \frac{C}{R^3}, 
\quad |\La\4{B}|\le\frac{C}{R^4},\quad|\nabla\eta_R|\le\frac{C}{R}, 
\quad |\La\eta_R|\le\frac{C}{R^2}
\end{equation}
and
\begin{equation*}
\left|\4{u}-\4{v}\right|\leq \left|\4{B}_{k_1}-\4{B}_{k_2}\right|\leq 
\frac{C}{R^4} 
\end{equation*}
in $B_{2R}(0)\bs B_R(0)$ for any $R>1$ and some constant $C>0$, by
\eqref{eq-aligned-estimate-first-I-1-1} 
\begin{equation}\label{I1-bound}
\left|I_{1,R}\right|\le C'\quad\forall R>1,s>-\log T.
\end{equation}
Similarly there exists a constant $C>0$ such that 
\begin{equation}\label{I2-bound}
|I_{2,R}|\le C\quad\forall R>1,s>-\log T
\end{equation} 
By \eqref{eq-Kato-inequality-2}, 
\eqref{I3-bound}, \eqref{I1-bound}, and \eqref{I2-bound},
\begin{equation*}\label{eq-Kato-inequality-21}
\frac{d}{ds}\int_{\R^N}|\4{u}-\4{v}|(x,s)\4{B}^{\alpha}(x)\eta_R(x)\,dx
\le C\quad\forall R>1, s>-\log T
\end{equation*}
for some constant $C>0$. Integrating the above differential inequality
and letting $R\to\infty$ we get \eqref{Barenblatt-L1-bound}
and the lemma follows.
\end{proof}

\begin{lem}\label{lem-weighted-L-1-contraction-1234}
Let $N\ge 5$ and let $\4{u}$, $\4{v}$, be two solutions of 
\eqref{eq-rescaled-main-equation-of-v} with initial values $\4{u}_0$,
 $\4{v}_0$, satisfying \eqref{eq-initial-tilde-u-trapped-by-barenblat} 
and $\4{u}_0-\4{v}_0\in L^1(\4{B}^{\alpha},\R^N)$ with 
$\alpha=\frac{N-4}{2}$. Let $\4{B}=\4{B}_{k_2}$. If
\begin{equation}\label{tidle-u0-v0-eqn}
\max_{\R^N}|\4{u}_0-\4{v}_0|\neq 0,
\end{equation}
then for any $s>-\log T$ there exist constants $C(s)>0$ and  $R_0>1$ such that
\begin{equation}\label{eq-tilde-u-tilde-v-strong-decreasing-12}
\left\|\left(\4{u}-\4{v}\right)(\cdot,s)\4{B}^{\alpha}\eta_R\right\|_{L^1(\R^N)}
<\left\|\left(\4{u}_0-\4{v}_0\right)(\cdot,s)\4{B}^{\alpha}
\eta_R\right\|_{L^1(\R^N)}-C(s)\quad\forall R\ge R_0
\end{equation}
where $\eta_R$ is as in the proof of 
Lemma \ref{lem-preliminary-esti-first-1}.
\end{lem}
\begin{proof}
We will use a modification of the proof of Lemma 4.1 of \cite{DS1} to 
prove the lemma. Let $\eta_R\in C^{\infty}_0(\R^N)$ be
as in the proof of Lemma \ref{lem-preliminary-esti-first-1}. Let $q=\4{u}-\4{v}$ and $\4{a}(x,s)$ be given by 
\eqref{eq-mean-value-thm-coeffi-a-1}. 
By the proof of Lemma \ref{lem-initial-L-tilde-B-to-L-tilde-B},
\eqref{eq-Kato-inequality-2} holds. Integrating 
\eqref{eq-Kato-inequality-2},
\begin{align*}
&\int_{\R^N}|q(x,s)|\4{B}^{\alpha}(x)\eta_R(x)\,dx-\int_{\R^N}|\4{u}_0-\4{v}_0|
\4{B}^{\alpha}(x)\eta_R(x)\,dx\nonumber\\
\le&\int_{-\log T}^{s}\int_{\R^N}\4{a}(x,\tau)|q|(x,\tau)
\left(\4{B}^{\alpha}\La\eta_R+\eta_R\La\4{B}^{\alpha}+2\nabla\eta_R\cdot
\nabla\4{B}^{\alpha}\right)\,dx\,d\tau\nonumber\\
&\qquad -\frac{1}{N-2}\int_{-\log T}^{s}\int_{\R^N}|q|(x,\tau)x
\cdot\left(\4{B}^{\alpha}\nabla\eta_R+\eta_R\nabla\4{B}^{\alpha}\right)
\,dx\,d\tau.
\end{align*}
Hence 
\begin{align}
&\int_{\R^N}|q(x,s)|\4{B}^{\alpha}(x)\eta_R(x)\,dx-\int_{\R^N}|\4{u}_0-\4{v}_0|\4{B}^{\alpha}(x)\eta_R(x)\,dx\nonumber\\
\le&\frac{C}{R^2}\int_{-\log T}^s\int_{R\leq |x|\leq 2R}
\4{a}(x,\tau)|q|(x,\tau)\4{B}^{\alpha}(x)\,dx\,d\tau\nonumber\\
&\qquad+\frac{C}{R}\int_{-\log T}^s\int_{R\leq |x|\leq 2R}
\4{a}(x,\tau)|q|(x,\tau)\left|\nabla\4{B}^{\alpha}(x)\right|\,dx\,d\tau\nonumber\\
&\qquad+C\int_{-\log T}^s\int_{R\leq |x|\leq 2R}|q|(x,\tau)
\4{B}^{\alpha}(x)\,dx\,d\tau\nonumber\\
&\qquad+\int_{-\log T}^s\int_{\R}\left\{a(x,\tau)\La
\4{B}^{\alpha}-\frac{1}{N-2}x\cdot\nabla\4{B}^{\alpha}\right\}|q|(x,\tau)\eta_R(x)\,dx\,d\tau\nonumber\\
=&I_{1,R}+I_{2,R}+I_{3,R}+I_{4R}\qquad\qquad\forall R>0,
s>-\log T.\label{diff-ineqn} 
\end{align}
Now by \eqref{rescaled-tilde-B-bound},
\begin{equation}\label{nable-B-bound}
|\nabla \4{B}^{\alpha}|\le CR^{-1}\4{B}^{\alpha}\quad\forall 
R\le |x|\le 2R, R>1
\end{equation}
for some constant $C>0$.
Then by \eqref{eq-mean-value-thm-coeffi-a-bound-lower-upper}
and \eqref{nable-B-bound},
\begin{equation}\label{I-comparison}
0\le I_{2,R}\le CI_{1,R}\le C'I_{3,R}\quad\forall R>0.
\end{equation} 
Since by Lemma \ref{lem-initial-L-tilde-B-to-L-tilde-B},
$$
\int_{-\log T}^{s}\int_{\R^N}|q|(x,\tau)\4{B}^{\alpha}(x)\,dx\,d\tau<\infty
\quad\forall s>-\log T,
$$
we have
\begin{equation}\label{I3-bound-2}
\lim_{R\to\infty}I_{3,R}=\lim_{R\to\infty}\int_{-\log T}^s
\int_{R\leq |x|\leq 2R}|q|(x,\tau)\4{B}^{\alpha}(x)\,dxd\tau=0.
\end{equation}
By \eqref{I-comparison} and \eqref{I3-bound-2},
\begin{equation}\label{I12-bound}
\lim_{R\to\infty}I_{1,R}=\lim_{R\to\infty}I_{2,R}=0.
\end{equation}
By \eqref{eq-direct-computation-La-tilde-B-alpha-nabla-tilde-1} and
\eqref{tidle-u0-v0-eqn} for any 
$s>-\log T$ there exist constants $C(s)>0$ and $R_1>1$ such that
\begin{equation}\label{I14-bound}
I_{4,R}<-C(s)\quad\forall R\ge R_1.
\end{equation}
By \eqref{diff-ineqn}, \eqref{I3-bound-2}, \eqref{I12-bound} and 
\eqref{I14-bound}, for any $s>-\log T$ there exists a constant $R_0>R_1$ 
such that \eqref{eq-tilde-u-tilde-v-strong-decreasing-12} holds
and the lemma follows.
\end{proof}

By Lemma \ref{lem-weighted-L-1-contraction-1234} and an argument similar 
to the proof of Lemma 1 of Osher and Ralston \cite{OR} but with the
$L^1$ norm there being replaced by the $L^1(\4{B}^{\alpha},\R^N)$ norm
we have the following result.

\begin{lem}(cf. Lemma 1 of \cite{OR})\label{Osher-and-Ralston-Lemma-1}
Let $N\ge 5$, $\alpha=(N-4)/2$, and $\4{B}_{k_0}$ be the rescaled Barenblatt
solution. Suppose $\|\4{u}(\cdot,s_i)-\4{w}_0\|_{L^1(\4{B}^{\alpha},\R^N)}
\to 0$ as $i\to\infty$. If $\4{w}$ is a solution 
of \eqref{eq-rescaled-main-equation-of-v} in $\R^N\times[0,\infty)$ 
with initial value $\4{w}(x,0)=\4{w}_0(x)$, then
\begin{equation*}
\|\4{w}(\cdot,s)-\4{B}_{k_0}\|_{L^1(\4{B}^{\alpha},\R^N)}
=\|\4{w}_0-\4{B}_{k_0}\|_{L^1(\4{B}^{\alpha},\R^N)}\quad\forall s>0.
\end{equation*}
\end{lem}

By an argument similar to the proof of Claim 4.4 of \cite{DS1} but
with Lemma \ref{lem-weighted-L-1-contraction-1234} and 
Corollary \ref{cor-tilde-u-L-1-contraction-987} replacing Lemma 4.1 
and Corollary 2.2 in the proof there we have the following result.

\begin{lem}\label{lem-claim-4-4-in-DS}
Let $N\ge 5$ and let $u_0$, $u$, $\4{u}$, $\4{u}_i$, $\4{u}_{i_k}$ and $\4{w}$ 
be as in Lemma \ref{lem-convergence-on-compact-subset-of-R-N=4}. 
Then the sequence $\4{u}_{i_k}(x,s)$ converges to 
$\4{w}(x,s)$ in $L^1(\4{B}^{\alpha},\R^N)$-norm as $k\to\infty$.
\end{lem}

Then by the same argument as the proof of Theorem 1.2 of \cite{DS1} on 
P.110 but with Lemma 4.2, Lemma 4.3, and Claim 4.4 there being replaced 
by Lemma \ref{Osher-and-Ralston-Lemma-1}, 
Lemma \ref{lem-convergence-on-compact-subset-of-R-N=4}
and Lemma \ref{lem-claim-4-4-in-DS}, we get 
Theorem \ref{lem-main-theorem-MS}. This completes the proof of 
Theorem \ref{lem-main-theorem-MS}.

\section{Improvement}
\setcounter{equation}{0}
\setcounter{thm}{0}

In this section we will improve Theorem~\ref{lem-main-theorem-MS} 
by removing the assumption $u_0\ge B_{k_1}(x,0)$ where $B_{k_1}(x,t)$ 
is given by \eqref{eq-self-similar-barenblett-solution} for some $T>0$. 
Let $T>0$ and $k_0>0$ be fixed constants. 
Denoting by 
\begin{equation*}
B_{k_0}(x,t)=\frac{2(N-2)(T-t)_+^{\frac{N}{N-2}}}{k_0+(T-t)_+^{\frac{2}{N-2}}|x|^2},
\end{equation*}
we will prove the following result.

\begin{thm}\label{thm-improvement-first-1}
Let $N\ge 5$. Suppose
\begin{equation}\label{eq-condition-of-improvement-0}
0\le u_0\le B_{k_0}(x,0)\quad\mbox{ in }\R^N
\end{equation} 
and 
\begin{equation}\label{eq-condition-of-improvement-1}
\left|u_0(x)-B_{k_0}(x,0)\right|\leq f(|x|)\in L^1(\R^N) 
\end{equation}
for some nonnegative radially symmetric function $f$. Then the maximal 
solution $u$ of \eqref{eq-cases-main-cauchy-problem} vanishes at 
the same time $T$ as $B_{k_0}(x,t)$ and the rescaled solution $\4{u}(x,s)$ 
given by \eqref{eq-rescaled-function} converges uniformly on $\R^N$ 
and in $L^1(\4{B}^{\frac{N-4}{2}},\R^N)$ as $s\to\infty$ 
to the rescaled Barenblatt solution $\4{B}_{k_0}$. 
\end{thm}

We will first prove that condition \eqref{eq-condition-of-improvement-1} 
implies the $L^1$-contraction principle.

\begin{lem}\label{lem-L-1-contraction-in-improvement-6-2-0008}
Let $N\ge 3$ and $0\le u_0$ satisfy \eqref{eq-condition-of-improvement-1}
for some function $0\le f\in L^1(\R^N)$. Suppose $u$ 
is the maximal solution of \eqref{eq-cases-main-cauchy-problem} in 
$\R^N\times(0,T_0)$ for some $T_0>0$. Then
\begin{equation}\label{eq-L-1-contraction-of-improvement-123}
\int_{\R^N}\left|u(\cdot,t)-B_{k_0}(\cdot,t)\right|\,dx\leq \|f\|_{L^1(\R^N)} 
\qquad \forall 0<t<\min\left(T,T_0\right).
\end{equation}
\end{lem}
\begin{proof}
For any $k\geq k_0$, let $u_k$ be the maximal solutions of 
\eqref{eq-cases-main-cauchy-problem} (cf.\cite{Hu2}) in 
$\R^N\times(0,T_k)$ with initial values
\begin{equation*}
u_{0,k}(x)=\max\left(B_k(x,0),u_0(x)\right), \qquad \forall k\geq k_0
\end{equation*}
where $T_k$ is the maximal time of extistence of the solution $u_k$. 
Since
\begin{equation*}
u_0(x)\leq u_{0,k'}(x)\leq u_{0,k}(x)
\le 2(N-2)k_0^{-1}T^{\frac{N}{N-2}}\quad \mbox{ in }\R^N\quad\forall 
k'\ge k\ge k_0
\end{equation*}
and $u$, $u_k$ are the maximal solutions of 
\eqref{eq-cases-main-cauchy-problem} with initial values $u_{0}$, $u_{0,k}$ 
respectively, by the result of \cite{Hu2}, 
\begin{equation}\label{eq-comparison-maximal-functions-u-u-k-B-k-0}
u(x,t)\leq u_{k'}(x,t)\leq u_{k}(x,t)\le 2(N-2)k_0^{-1}T^{\frac{N}{N-2}}
<\infty\quad\mbox{ in }\R^N\times(0,T_0)
\end{equation}
for any $k'\ge k\ge k_0$. Then $T_k\geq T_{k'}\geq T_0$ for all 
$k'\geq k\geq k_0$. Hence the equation \eqref{eq-cases-main-cauchy-problem} 
for the sequence $\{u_k\}_{k\geq k_0}$ is uniformly parabolic on any compact 
subset of $\R^N\times(0,T_0)$. By the Schauder estimates \cite{LSU}, 
$\{u_k\}_{k\geq k_0}$ is equi-H\"older continuous on any compact subset 
of $\R^N\times(0,T_0)$. Since the sequence of solution $\{u_k\}_{k\geq k_0}$ 
is decreasing as $k\to\infty$ and bounded below by $u$, $u_{k}$ 
converges uniformly to a solution $v$ of 
\eqref{eq-cases-main-cauchy-problem} on every compact subset of 
$\R^N\times(0,T_0)$ as $k\to\infty$. By an argument similar to the proof 
of Theorem 2.4 in \cite{Hu2}, $v$ has initial value $u_0$. 
Letting $k\to\infty$ in \eqref{eq-comparison-maximal-functions-u-u-k-B-k-0},
\begin{equation*}
v(x,t)\ge u(x,t)\quad \mbox{in $\R^N\times(0,T_0)$}.
\end{equation*}  
On the other hand since $u$ is the maximal solution of with initial value 
$u_{0}$,
\begin{equation*}
u(x,t) \ge v(x,t) \qquad \mbox{in $\R^N\times(0,T_0)$}.
\end{equation*}
Hence $u=v$ on $\R^N\times(0,T_0)$. Since both $B_{k_0}\ge B_k$ and 
$u_k\ge B_k$ for any $k\geq k_0$, by Lemma~\ref{eq-L-1-contraction-principle},
\begin{equation}
\label{eq-L-1-contraction-maximal-function-of-u-k-in-Section-improvement}
\int_{\R^N}|B_{k_0}-u_k|(x,t)\,dx \leq \int_{\R^N}|B_{k_0}-u_k|(x,0)\,dx
\leq \|f\|_{L^1(\R^N)}, \quad \forall k\geq k_0,\,\,
\forall 0<t<\min\left(T,T_0\right).
\end{equation}
Letting $k\to\infty$ in 
\eqref{eq-L-1-contraction-maximal-function-of-u-k-in-Section-improvement}, 
we get \eqref{eq-L-1-contraction-of-improvement-123} and the lemma follows.
\end{proof}

Note that if $0\le u_0\in L^{\infty}(\R^N)$ satisfies 
\eqref{eq-condition-of-improvement-1} for some function $0\le f\in L^1(\R^N)$, 
then the maximal solution $u$ of \eqref{eq-cases-main-cauchy-problem} and $B_{k_0}$ 
have the same vanishing time. The reason is as follows. 
Let $T_0>0$ be the maximal time of existence of the solution 
$u$ of \eqref{eq-cases-main-cauchy-problem}. We first suppose that $T_0<T$, 
then by \eqref{eq-L-1-contraction-of-improvement-123}
\begin{equation*}
\int_{\R^N}|B_{k_0}(x,T_0)|\,dx \leq \|f\|_{L^1(\R^N)}.
\end{equation*}
On the other hand, since the dimension $N\geq 3$, $B_{k_0}(x,T_0)
\notin L^1(\R^N)$. Contradiction arises. Hence $T_0\geq T$. We now 
assume that $T_0>T$. Letting $t\nearrow T$ in 
\eqref{eq-L-1-contraction-of-improvement-123},
\begin{equation*}
\int_{\R^N}|u(x,T)|\,dx \leq \|f\|_{L^1(\R^N)}.
\end{equation*}
This contradicts the result of Vazquez \cite{V1} which said that 
\eqref{eq-cases-main-cauchy-problem} has no solution that is in $L^1(\R^N)$. 
Hence $T=T_0$ and the maximal solution $u$ vanishes at the same time as 
$B_{k_0}(x,t)$.

We next prove a lemma on the existence of maximal solutions of 
\eqref{eq-cases-main-cauchy-problem}.

\begin{lem}[cf. Corollary 2.8 in \cite{Hu2}]
\label{lem-cf-Corollary-2-8-in-hui-1-in-improvement}
Let $N\ge 3$ and let $g(x)=B_{k_0}(x,0)-h(x)$ for some radially symmetric
function $0\le h\in L^{\infty}(\R^N)\cap L^1(\R^N)$ be such that $g(x)\ge 0$ 
on $\R^N$. Then there exists a unique maximal solution $u$ of 
\eqref{eq-cases-main-cauchy-problem} in $\R^N\times(0,T)$ with initial value 
$g$.
\end{lem}
\begin{proof}
Since the proof is similar to that of Corollary 2.8 of \cite{Hu2}, we will 
only give a sketch of the proof here. For any $R>0$ and any function 
$\psi\in L^1(B_R(0))$, let
\begin{equation*}
\4{G}_R(\psi)(x)=\int_{B_R(0)}\left(G_R(0,y)-G_R(x,y)\right)\,\psi(y)\,dy 
\qquad \forall |x|\leq R
\end{equation*}
where $G_R$ is the Green function for the Laplacian on $B_R(0)$. Since $u_0$ 
is radially symmetric and 
\begin{equation*}
B_{k_0}(x,0)\geq \frac{C_1}{|x|^2} \qquad \forall |x|\geq 1
\end{equation*}
for some constant $C_1>0$, for any $R>1$, we have (cf.\cite{Hu2})
\begin{equation*}
\4{G}_R(g)(x)=\int_{0}^{|x|}\frac{1}{\omega_{N}r^{N-1}}
\left(\int_{|y|\le r}g(y)\,dy\right)\,dr\geq 0, \qquad \forall |x|\leq R
\end{equation*}
and
\begin{equation*}
\begin{aligned}
\4{G}_R(g)(x)&\geq\int_{1}^{|x|}\frac{1}{\omega_{N}r^{N-1}}
\left(\int_{1\leq |y|\leq r}g(y)\,dy\right)\,dr\\
&\geq \frac{C_1}{N-2}\log |x|-\left(\frac{C_1}{(N-2)^2}
+\frac{\|h\|_{L^1(\R^N)}}{(N-2)\omega_N}\right)\left(1-|x|^{2-N}\right) 
\qquad \forall 1\leq |x|\leq R.
\end{aligned}
\end{equation*}
Hence there exist constants $R_1>1$ and $C_2>0$ such that 
\begin{equation}\label{eq-calculus-of-tilde-green-in-improve-545}
\4{G}_R(g)(x)\ge C_2\log |x|\qquad \forall R_1
\le |x|\leq R.
\end{equation}
Then by \eqref{eq-calculus-of-tilde-green-in-improve-545} and the result of 
\cite{Hu2}, \eqref{eq-cases-main-cauchy-problem} has a unique maximal 
solution $u$ with initial value $g$ in $\R^N\times(0,T_1)$ for some 
constant $T_1>0$. Since the solution $u$ is unique and $g$ is radially 
symmetric, $u(\cdot,t)$ is radially symmetric in $\R^N\times(0,T_1)$. 
Let $T_2>0$ be the maximal time of existence of the solution $u$. By the 
discussion just before the lemma we have $T_2=T$ and the lemma follows.
\end{proof}

By Lemma 1.8 of \cite{Hu2}, Lemma~\ref{lem-cf-Corollary-2-8-in-hui-1-in-improvement}, 
and an argument similar to the proof of Corollary 2.8 of \cite{Hu2} 
we have the following corollary. 

\begin{cor}
\label{lem-cf-Corollary-2-8-in-hui-1-in-improvement}
Let $N\ge 3$ and let $B_{k_0}(x,0)-h(x)\le u_0(x)\le B_{k_0}(x,0)$ for 
some radially symmetric function $h\in L^{\infty}(\R^N)\cap L^1(\R^N)$ 
satisfying $0\le h(x)\le B_{k_0}(x,0)$ on $\R^N$. Then there exists 
a unique maximal solution $u$ of \eqref{eq-cases-main-cauchy-problem} in 
$\R^N\times(0,T)$ satisfying $0\le u(x,t)\le B_{k_0}(x,t)$ in $\R^N\times 
(0,T)$ with initial value $u_0$.
\end{cor}

\begin{lem}\label{improve-tidle-u-bounded-above-below}
Let $N\ge 3$. Suppose $u_0$ satisfies \eqref{eq-condition-of-improvement-0},
\eqref{eq-condition-of-improvement-1}, and $u$ is the maximal 
solution of \eqref{eq-cases-main-cauchy-problem}. Then there exist positive 
constants $C_1$, $C_2$, $C_3$, $r_0$, $s_0$ such that the rescaled 
function $\4{u}$ given by 
\eqref{eq-rescaled-function} satisfies
\begin{equation}\label{eq-upper-lower-bound-of-tilde-u-under-new-condition}
C_1\frac{e^{-C_3e^s\|f\|_{L^1}}}{1+r^2}\leq \4{u}(r,s) 
\leq C_2\frac{e^{C_3e^s\|f\|_{L^1}}}{1+r^2}  \qquad \forall r
\geq r_0,\,\,s\geq s_0.
\end{equation}
\end{lem}
\begin{proof}
We will use a modification of the proof of Proposition 6.2 of \cite{DS1}
to prove the lemma. We will first prove 
\eqref{eq-upper-lower-bound-of-tilde-u-under-new-condition} 
under the assumption that $u_0$ is radially symmetric. Let 
$\{u_{k}\}_{k\geq k_0}$ be the sequence constructed in the proof of 
Lemma \ref{lem-L-1-contraction-in-improvement-6-2-0008}. As observed 
in the proof of Lemma \ref{lem-before-L-1-contraction-46} the function
\begin{equation*}
w_k(x)=\int_l^t|\log u_k-\log B_{k_0}|(x,\tau)\,d\tau\qquad
\forall 0<l<t<T
\end{equation*}
satisfies
\begin{equation*}
\La w_k(x)\geq -|u_k-B_{k_0}|(x,l) \qquad \mbox{in $\R^N$}\quad\forall
0<l<t<T
\end{equation*}
and
\begin{equation*}
\La(w_k-Z_k(\cdot,l))\geq 0 \qquad \mbox{in $\R^N$}\qquad\qquad\qquad
\forall 0<l<t<T
\end{equation*}
where $Z_k(x,l)$ is given by
\begin{equation*}
Z_k(x,l)=\int_r^{\infty}\frac{1}{\omega_N\rho^{N-1}}
\int_{|y|\leq\rho}|u_k-B_{k_0}|(y,l)\,dyd\rho, \qquad r=|x|.
\end{equation*}
Note that $Z_k$ satisfies 
$\La Z_k(x,l)=-|u_k-B_{k_0}|(x,l)$ in $\R^N$. Then, as in the proof of 
Lemma \ref{lem-before-L-1-contraction-46},
\begin{equation*}
w_k(x)\leq Z_k(x,l) \qquad \mbox{in $\R^N$}.
\end{equation*}
Thus
\begin{equation}\label{eq-conclude-that-in-improvement-of-proposition-1110}
w_k(x)\leq C_3\frac{\|(u_k-B_{k_0})(l)\|_{L^1(\R^N)}}{r^{N-2}}, \qquad r=|x|\geq 1,
\end{equation}
for some constant $C_3>0$. By \eqref{eq-condition-of-improvement-1}, 
\eqref{eq-L-1-contraction-maximal-function-of-u-k-in-Section-improvement} 
and \eqref{eq-conclude-that-in-improvement-of-proposition-1110},
\begin{equation}\label{eq-upper-lower-bound-of-int-u-in-improvements-111}
\int_{l}^{t}\log B_{k_0}(x,\tau)\,d\tau-C_3\frac{\|f\|_{L^1(\R^N)}}{r^{N-2}}
\leq \int_{l}^{t}\log u_k(x,\tau)\,d\tau \leq \int_{l}^{t}
\log B_{k_0}(x,\tau)\,d\tau+C_3\frac{\|f\|_{L^1(\R^N)}}{r^{N-2}}
\end{equation}
holds for any $r=|x|\ge 1$ and $0<l<t<T$. We now let $t\in [3T/4,T)$ and choose 
$l\in [T/2,T)$ such that $T-t=t-l$. For any $l\leq \tau\leq t$,
\begin{equation*}
B_{k_0}(x,\tau)
\leq \frac{2(N-2)(T-l)_+^{\frac{N}{N-2}}}{k_0+(T-t)_+^{\frac{2}{N-2}}|x|^2}
=\frac{2(N-2)\left[2(T-t)_+\right]^{\frac{N}{N-2}}}{k_0+(T-t)_+^{\frac{2}{N-2}}|x|^2}
=2^{\frac{N}{N-2}}B_{k_0}(x,t)
\end{equation*}
and similarly 
\begin{equation*}
B_{k_0}(x,\tau)\geq 2^{-\frac{N}{N-2}}B_{k_0}(x,l)\quad\forall l<\tau<T.
\end{equation*}
Hence
\begin{equation}\label{eq-upper-lower-bound-of-int-u-in-improvements-222}
(t-l)\left\{\log B_{k_0}(x,l)-\log 2^{\frac{N}{N-2}}\right\}
\leq \int_{l}^{t}\log B_{k_0}(x,\tau)\,d\tau\leq (t-l)\left\{\log B_{k_0}(x,t)
+\log 2^{\frac{N}{N-2}}\right\}.
\end{equation}
By \eqref{eq-upper-lower-bound-of-int-u-in-improvements-111} and 
\eqref{eq-upper-lower-bound-of-int-u-in-improvements-222},
\begin{equation}\label{eq-last-two-inequalities-for-estimate-of-u-1}
\log \left(\frac{B_{k_0}(x,l)}{C_4}\right)
\leq \frac{1}{t-l}\int_{l}^{t}\log u_k(x,\tau)\,d\tau 
\leq \log \left(C_4B_{k_0}(x,t)\right) \qquad \forall |x|=r\ge 1
\end{equation}
where $C_4=e^{C_3\frac{\|f\|_{L^1}}{T-t}}2^{\frac{N}{N-2}}$. Since $u_k$ satisfies the 
Aronson-Benilan inequality (cf. \cite{Hu2}), 
$$
u_t\le\frac{u}{t}\qquad\qquad\mbox{ in }\R^N\times (0,T), 
$$
we have
\begin{align}\label{eq-last-two-inequalities-for-estimate-of-u-2}
&\frac{\tau}{t}u_k(x,t)\le u_k(x,\tau)\le\frac{\tau}{l}u_k(x,l)
\qquad\qquad\qquad\qquad\qquad\qquad\qquad\forall x\in\R^N,
l\le\tau\le t\nonumber\\
\Rightarrow\quad&\log\left(\frac{l}{t}\,\,u_k(x,t)\right)\leq \frac{1}{t-l}
\int_{l}^{t}\log u_k(x,\tau)\,d\tau \leq \log\left(\frac{t}{l}\,\,u_k(x,l)
\right)\quad\forall  x\in\R^N,l\le\tau\le t.
\end{align}
Now by our choice for $l$ we have $t/l\le 2$. Then  
by \eqref{eq-last-two-inequalities-for-estimate-of-u-1} and
\eqref{eq-last-two-inequalities-for-estimate-of-u-2},
\begin{equation}\label{eq-last-two-inequalities-for-estimate-of-u-11}
u_k(x,t)\leq C_5e^{C_3\frac{\|f\|_{L^1}}{T-t}}B_{k_0}(x,t) 
\qquad \forall\,\, |x|\ge 1,3T/4\le t<T
\end{equation}
and
\begin{equation}\label{eq-last-two-inequalities-for-estimate-of-u-22}
u_k(x,l)\geq C_6e^{-C_3\frac{\|f\|_{L^1}}{T-t}}B_{k_0}(x,l) 
\qquad \forall\,\, |x|\ge 1,3T/4\le l<T
\end{equation}
for some constants $C_5$, $C_6>0$. Letting $k\to\infty$ in 
\eqref{eq-last-two-inequalities-for-estimate-of-u-11} and 
\eqref{eq-last-two-inequalities-for-estimate-of-u-22},
\begin{equation}\label{eq-last-two-inequalities-for-estimate-of-u-33}
C_6e^{-C_3\frac{\|f\|_{L^1}}{T-t}}B_{k_0}(x,t)\leq u(x,t)
\leq C_5e^{C_3\frac{\|f\|_{L^1}}{T-t}}B_{k_0}(x,t)\qquad\forall\,\,|x|\ge 1,
3T/4\le t<T.
\end{equation}
By \eqref{eq-last-two-inequalities-for-estimate-of-u-33} we conclude 
after rescaling,
\begin{equation*}
C_1\frac{e^{-C_3e^s\|f\|_{L^1}}}{1+r^2}\leq \4{u}(r,s) 
\leq C_2\frac{e^{-C_3e^s\|f\|_{L^1}}}{1+r^2}, \qquad 
\forall r\ge T^{1/(N-2)},s\ge -\log (T/4)
\end{equation*}
for some constants $C_1>0$, $C_2>0$.

When $u_0(x)$ is nonradial and satisfies \eqref{eq-condition-of-improvement-1}, 
by the above result for the radially symmetric initial data case and an argument 
similar to the last step of the proof of Proposition 6.2 of \cite{DS1} on 
p.118 of \cite{DS1} the lemma follows.
\end{proof}

\begin{proof}[\textbf{Proof of Theorem \ref{thm-improvement-first-1}}]

By Lemma~\ref{improve-tidle-u-bounded-above-below}
there exist positive constants $C_1$, $C_2$, $C_3$, $s_0$ and $r_0$ such 
that \eqref{eq-upper-lower-bound-of-tilde-u-under-new-condition} holds. 
Let $s_1>s_0>-\log T$ and $Q_{r_0}^{s_1}=B_{r_0}(0)\times(s_0,s_1)$. Then there 
exist constants $C_4>0$, $C_5>0$ such that
\begin{equation}\label{new-bd}
\frac{C_4}{1+r_0^2}\leq\4{u}(x,s)\le\frac{C_5}{1+r_0^2}
\end{equation}
on the parabolic boundary $\partial_p Q_{r_2}^{s_1}=(\2{B_{r_0}(0)}
\times\{s_0\})\cup (\partial B_{r_0}(0)\times(s_0,s_1))$. 
By the maximum principle,
\begin{equation}\label{eq-lower-bound-by-maximum-principle}
\frac{C_4}{1+r_0^2}\le\4{u}(x,s)\le\frac{C_5}{1+r_0^2} \qquad 
\mbox{in }Q_{r_0}^{s_1}.
\end{equation}
By \eqref{new-bd} and \eqref{eq-lower-bound-by-maximum-principle}, 
\begin{equation*}
\frac{C_4'}{1+|x|^2}\le\4{u}(x,s)\le\frac{C_5'}{1+|x|^2}\qquad 
\mbox{on }\R^N\times[s_0,s_1)
\end{equation*}
for some constants $C_4'>0$, $C_5'>0$. Hence $\4{u}-\4{B}_{k_0}$ satisfies
\eqref{eq-difference-between-tilde-u-and-tilde-v}  in $\R^N\times 
[s_0,\infty)$ with
\begin{equation*}
\4{a}(x,s)=\int_0^1\frac{d\theta}{\theta\4{u}+(1-\theta)\4{B}_{k_0}}
\end{equation*} 
and $\4{a}$ satisfies \eqref{eq-condition-of-Lemma-4-first-437}
for some constants $C_1>0$ and $C_2>0$. By 
\eqref{eq-condition-of-improvement-0},
\begin{equation*}
\4{u}(x,s)\le\frac{2(N-2)}{k_0+|x|^2}\quad\mbox{ in }\R^N\times [s_0,\infty).
\end{equation*}
Hence 
\begin{equation*}
\4{a}(x,s)\ge\frac{k_0+|x|^2}{2(N-2)}\quad\mbox{ in }
\R^N\times[s_0,\infty).
\end{equation*}
Then by an argument similar to the the proof of 
Theorem~\ref{lem-main-theorem-MS} in 
section~\ref{section-The-Non-Integral-Case-I-N-geq-4} the theorem follows.
\end{proof}

{\bf Acknowledgement} Sunghoon Kim was supported by Basic Science Research Program through the National Research Foundation of Korea(NRF) funded by the Ministry of Education, Science and Technology(2011-0030749).


\begin{thebibliography}{99}

\bibitem[A]{A} D.G.~Aronson, {\em The porous medium equation, CIME Lectures}, 
in Some problems in Nonlinear Diffusion, Lecture Notes in Mathematics 1224, 
Springer-Verlag, New York, 1986.

\bibitem[BBDGV]{BBDGV} A.~Blanchet, M.~Bonforte, J.~Dolbeault, G.~Grillo and 
J.L.~Vazquez, {\em Asymptotics of the fast diffusion equaiton via entropy 
estimates}, Arch. Rat. Mech. Anal. 191 (2009), 347--385.

\bibitem[BDGV]{BDGV} M.~Bonforte, J.~Dolbeault, G.~Grillo and J.L.~Vazquez,
{\em Sharp rates of decay of solutions to the nonlinear fast diffusion 
equation via functional inequalities}, Proc. Nat. Acad. Sci. 107 (2010), 
no. 38, 16459--16464.

\bibitem[BGV]{BGV} M.~Bonforte, G.~Grillo and J.L.~Vazquez, {\em Special
fast diffusion with slow asymptotics: entropy method and flow on a Riemannian
manifold}, Arch. Rat. Mech. Anal. 196 (2010), 631--680.

\bibitem[DK]{DK} P.~Daskalopoulos and C.E.~Kenig, {\em Degenerate 
diffusion-initial value problems and local regularity theory}, 
Tracts in Mathematics 1, European Mathematical Society, 2007.

\bibitem[DP1]{DP1} P.~Daskalopoulos and M.A.~Del Pino, {\em On a singular 
diffusion equation}, Comm. in Analysis Geometry 3 (1995), no. 3, 523--542.

\bibitem[DP2]{DP2} P.~Daskalopoulos and M.A.~del Pino, {\em Type II 
collapsing of maximal solutions to the Ricci flow in $\R^2$}, Ann. 
Inst. H. Poincar\'e Anal. Non Linaire 24 (2007), 851--874. 

\bibitem[DS1]{DS1} P.~Daskalopoulos and N.~Sesum, {\em On the extinction 
profile of solutions to fast diffusion}, J. Reine Angew. Math. 622 (2008), 
95--119.

\bibitem[DS2]{DS2} P.~Daskalopoulos and N.~Sesum, {\em Type II extinction 
profile of maximal solutions to the Ricci flow equation}, J. Geom. Anal. 
20 (2010), 565--591.

\bibitem[G]{G} P.~G.~de Gennes, {\em Wetting: statics and dynamics},
Reviews of Modern Physics 57 (1985), no. 3, 827--863.

\bibitem[GT]{GT} D.~Gilbarg and N.S.~Trudinger, {\em Elliptic partial 
differential equations of second order}, 2nd. edition, 
Springer-Verlag, Berlin, 1983.

\bibitem[Hs1]{Hs1} S.Y.~Hsu, {\em Asymptotic profile of solutions of a 
singular diffusion equation as $t\to\infty$}, Nonlinear Analysis TMA 
48 (2002), 781--790.

\bibitem[Hs2]{Hs2} S.Y.~Hsu, {\em Asymptotic behaviour of solutions of 
the equation $u_t=\Delta\log u$ near the extinction time}, Advances in 
Differential Equations 8 (2003), no. 2, 161--187.

\bibitem[Hs3]{Hs3} S.Y.~Hsu, {\em Behaviour of solutions of a singular 
diffusion equation near the extinction time}, Nonlinear Analysis TMA 56 
(2004), no. 1, 63--104.

\bibitem[Hu1]{Hu1} K.M. Hui, Existence of solutions of the equation 
$u_t=\Delta\log u$  Nonlinear Analysis TMA 37 (1999), 875--914. 

\bibitem[Hu2]{Hu2} K.M.~Hui, {\em On Some Dirichlet and Cauchy Problems 
for a Singular Diffusion Equation,} Differential Integral Equations 15
(2002), no. 7, 769--804.

\bibitem[Hu3]{Hu3} K.M.~Hui, {\em Collapsing behaviour of a singular 
diffusion equation}, Discrete and Continuous Dynamical 
Systems-Series A  32 (2012), no. 6, 2165-2185. 

\bibitem[K]{K} T.~Kato, {\em Schr\"odinger operators with singular 
potentials}, Israel J. Math. 13 (1973), 135--148.

\bibitem[LSU]{LSU} O.A.~Ladyzenskaya, V.A.~Solonnikov and N.N.~Uraltceva 
{\em Linear and Quasilinear Equations of Parabolic Type,} Transl. Math. 
Mono. vol. 23, Amer. Math. Soc., Providence, R.I., USA, 1968.

\bibitem[LT]{LT} P.L.~Lions and G.~Toscani, {\em Diffusive limit for
finite velocity Boltzmann kinetic models}, Revista
Matematica Iberoamericana 13 (1997), no. 3, 473--513.

\bibitem[OR]{OR} S.J.~Osher and J.V.~Ralston, {\em $L^1$ stability of 
traveling waves with applications to convective porous media flow,} 
Comm. Pure Appl. Math. 35 (1982), 737--749.

\bibitem[P]{P} L.A.~Peletier, {\em The porous medium equation} in
Applications of Nonlinear Analysis in the Physical Sciences,
H.~Amann, N.~Bazley, K.~Kirchgassner editors, Pitman, Boston, 1981.

\bibitem[V1]{V1} J.L.~Vazquez, {\em Nonexistence of solutions for nonlinear 
heat operators of fast-diffusion type,} J. Math. Pures Appl. 71 (1992), 
503-526.

\bibitem[V2]{V2} J.L.~Vazquez, {\em Smoothing and decay estimates for 
nonlinear parabolic equations of porous medium type},
Oxford University Press Inc., New York, 2006.

\bibitem[V3]{V3} J.L.~Vazquez, {\em The porous medium equation Mathematical
Theory}, Oxford University Press Inc., New York, 2007.

\bibitem[WD]{WD} M.~B.~Williams and S.~H.~Davis, {\em Non-linear theory
of film rupture}, J. Colloidal and Interface Sc. 90 (1982), no. 1,
220--228.

\bibitem[W1]{W1} L.~F.~Wu, {\em The Ricci flow on complete $R^2$},
Comm. Anal. Geom. 1 (1993), 439--472.

\bibitem[W2]{W2} L.~F.~Wu, {\em A new result for the porous
medium equation}, Bull. Amer. Math. Soc. 28 (1993), 90--94.

\end{thebibliography}
\end{document}